\documentclass{article}

\newcommand{\negeer}[1]{}
\def\notintp{\mathrel{/\kern-0.6em\rhd}} 
\usepackage{amssymb}
\usepackage{amsfonts}
\usepackage{latexsym}
\usepackage{epsfig}
\usepackage{psfig}
\usepackage{xspace}
\usepackage{color}
\usepackage{amstext}
\usepackage{rotating}
\usepackage{amsmath}
\usepackage{amsthm}
\usepackage{import}
\usepackage{proof}	
\usepackage[english,dutch]{babel}
\usepackage{makeidx}
\newcommand{\inty}{interpretability\xspace}

\newcommand{\Inty}{Interpretability\xspace}
\newcommand{\rat}{reasonable arithmetical theories\xspace}

\newcommand{\il}{{\ensuremath{\textup{\textbf{IL}}}}\xspace}
\newcommand{\extil}[1]{\ensuremath{\textup{\textbf{IL}}{\sf\ensuremath{#1}}}\xspace}
\newcommand{\gl}{{\ensuremath{\textup{\textbf{GL}}}}\xspace}

\newcommand{\intl}[1]{{\ensuremath {\textup{\textbf{IL}}}({\rm #1})}}
\newcommand{\provl}[1]{{\ensuremath {\textup{\textbf{PL}}}({\rm #1})}}
\newcommand{\ilm}{\extil{M}}
\newcommand{\ilp}{\extil{P}}
\newcommand{\ilw}{\extil{W}}

\newcommand{\ilmo}{\extil{M_0}}
\newcommand{\ilwstar}{\extil{W^*}}

\newcommand{\ilal}{\intl{All}\xspace}
\newcommand{\principle}[1]{\formal{#1}}
\newcommand{\formal}[1]{\ensuremath{{\sf {#1}}\xspace}}

\newcommand{\prop}{\formal{Prop}\xspace}
\newcommand{\mcs}{\ensuremath{\textup{MCS}}\xspace}
\renewcommand{\principle}[1]{\formal{#1}}
\newcommand{\principel}[1]{\formal{#1}}

\newcommand{\trans}[2]{{\ensuremath{{#1}^{#2}}\xspace}}
\newcommand{\pair}[1]{{\ensuremath{\langle #1 \rangle}}\xspace}

\newcommand{\eqbydef}{:=}
\newcommand{\equivbydef}{:\Leftrightarrow}
\newcommand{\formil}{\ensuremath{{\sf Form}_{\il}} \xspace}

\newcommand{\depth}[1]{{\ensuremath{{\sf depth}(#1)}}\xspace}

\newcommand{\sucs}{\prec}
\newcommand{\crit}[1]{\sucs_{#1}}
\newcommand{\boxin}{\subseteq_{\Box}}
\newcommand{\crone}[2]{{\ensuremath{\mathcal{C}^{#1}_{#2}}}\xspace}
\newcommand{\geone}[2]{{\ensuremath{\mathcal{G}^{#1}_{#2}}}\xspace}
\newcommand{\adset}[1]{{\ensuremath{\mathcal{#1}}\xspace}}
\newcommand{\ext}{\subseteq}
\newcommand{\mo}{\principle{M_0}\xspace}
\newcommand{\comp}{;}
\newcommand{\spure}[2]{\mathcal{#2}_{#1}\xspace}
 \theoremstyle{plain}
 \newtheorem{theorem}{Theorem}[section]
 \newtheorem{lemma}[theorem]{Lemma}

 \newtheorem{corollary}[theorem]{Corollary}

 \theoremstyle{definition}
 \newtheorem{definition}[theorem]{Definition}
 
 \newtheorem{claim}{Claim}

 \newtheorem{remark}[theorem]{Remark}

\renewcommand{\trans}[1]{{#1}^{\textrm tr}}

\title{Modal Matters in Interpretability Logics}
\author{E. Goris\\ and\\ J. J. Joosten}
\date{2008}

\begin{document}

\maketitle

\begin{abstract}
This paper from 2008 is the first in a series of three related papers on modal methods in interpretability logics and applications. In this first paper the fundaments are laid for later results. These fundaments consist of a thorough treatment of a construction method to obtain modal models. This construction method is used to reprove some known results in the area of interpretability like the modal completeness of the logic \il.
Next, the method is applied to obtain new results: the modal completeness of the logic \ilmo, and modal completeness of \ilwstar.\end{abstract}

%\newpage
%\tableofcontents
%\newpage
%
\section{Introduction}
\Inty logics are primarily used to describe structural behavior of 
\inty between formal mathematical theories. We shall see that the logics come with
a good modal semantics that naturally extends the regular modal semantics giving it
a dynamical flavor. In this introduction we shall informally describe the project of
this paper. Formal definitions are postponed to later sections.

The notion of \inty that we are primarily interested in, is the notion of relativized 
\inty as studied e.g.\ by Tarski et al in \cite{tars:unde53}. Roughly, a theory $U$ interprets a
theory $V$ --we write $U\rhd V$-- if $U$ proves all theorems of $V$ under some 
structure preserving translation.
We allow for relativization of quantifiers. It is defendable to say that $U$ is as least as
strong as $V$ if $U\rhd V$. We think that it is clear that interpretations are 
worth to be studied, as they are omnipresent in both mathematics and meta-mathematics (Langlands 
Program, relative consistency proofs, undecidability results, Hilberts Programme and so forth).

One approach to the study of \inty is to study general structural behavior of 
\inty. An example of such a structural rule is the transitivity of 
\inty. That is, for any $U$, $V$ and $W$ we have that
if $U\rhd V$ and $V\rhd W$, then also $U\rhd W$. As we shall see, modal \inty logics 
provide an informative way to support this structural study. 
\Inty logics, in a sense, generate all structural rules.
Many important 
questions on \inty logics have been settled. One of the most prominent open 
questions at this time is the question of the \inty logic of all \rat. In this
paper we make a significant contribution to a solution of this problem. However, a modal
characterization still remains an open question.

The main aim of this paper is to establish some modal techniques/toolkit for 
\inty logics. Most techniques are aimed at establishing modal completeness results.
As we shall see, in the field of \inty logics, modal completeness can be a sticky 
business compared to unary modal logics. In this paper we make a first attempt
at pulling some (more) thorns out. Significant progress with this respect has also been made 
by de Jong and Veltman \cite{JoVe90}. 

We have a feeling that the general modal theory of \inty logics
is getting more and more mature. For example, fixed point phenomena and interpolation
are quite well understood (\cite{dejo:expl91}, \cite{arec:inte98}, \cite{vis97}).

Experience tells us that our modal semantics is quite informative and perspicuous. 
It is even the case that new arithmetical principles can be obtained from modal semantical
considerations. An example is our new principle \principel{R}. We found this principle 
primarily by modal investigation. Thus, indeed, there is a close match between 
the modal part and the arithmetical part. It is even possible to embed our modal 
semantics into some category of models of
arithmetic.

Although this paper is mainly a modal investigation, the main questions are still 
inspired by the arithmetical meaning of our logics. Thus, our investigations will lead
to applications concerning arithmetically informative notions like, essentially 
$\Sigma_1$-sentences, self provers and the \inty logic of all \rat.

\section{Interpretability logics}
In this section we will define the basic notions that are needed throughout the paper. We advise the reader to just skim through this section and use it to look up definitions whenever they are used in the rest of the paper.

\subsection{Syntax and conventions}
In this paper we shall be mainly interested in interpretability logics, the 
formulas of which, we write \formil, are defined as follows.
\[
\formil \eqbydef \bot \mid \prop \mid (\formil \rightarrow \formil) \mid (\Box \formil)
\mid (\formil \rhd \formil) 
\]

Here \prop is a countable set of propositional 
variables $p,q,r,s,t,p_0,p_1,\ldots$.
We employ the usual definitions of the logical operators
$\neg , \vee , \wedge$ and $\leftrightarrow$. Also shall we write 
$\Diamond \varphi$ for $\neg \Box \neg \varphi$. 
Formulas that start with a $\Box$ are called box-formulas or
$\Box$-formulas. Likewise we talk of $\Diamond$-formulas.

From now on we will
stay in the realm of interpretability logics. Unless 
mentioned otherwise, formulas or sentences are formulas of \formil. 
We will write $p\in \varphi$ to indicate that the proposition variable 
$p$ does occur in $\varphi$.
A literal is either a propositional variable or the negation of
a propositional variable.

In writing formulas we
shall omit brackets that are superfluous according to the following reading
conventions. We say that the operators $\Diamond$, $\Box$ and $\neg$ bind
equally strong. They bind stronger than the equally strong binding 
$\wedge$ and $\vee$ which in turn bind stronger than $\rhd$. The weakest 
(weaker than $\rhd$) binding
connectives are $\rightarrow$ and $\leftrightarrow$. We shall also omit outer
brackets. Thus, we shall write 
$A \rhd B \rightarrow A \wedge \Box C \rhd B \wedge \Box C$ instead of
$((A \rhd B) \rightarrow ((A \wedge (\Box C)) \rhd( B \wedge (\Box C))))$.

A schema of interpretability logic is syntactically like a formula. They are
used to generate formulae that have a specific form. We will not be specific
about the syntax of schemata as this is similar to that of formulas. Below,
one can think of $A$, $B$ and $C$ as place holders.

The rule of Modus Ponens
allows one to conclude $B$ from premises $A \rightarrow B$ and $A$. The 
rule of Necessitation allows one to conclude $\Box A$ from the premise
$A$.

\begin{definition}\label{defi:il}
The logic \il is the smallest set of formulas being closed under
the rules of Necessitation and of Modus Ponens, that contains
all tautological formulas and all instantiations of the following
axiom schemata.

\begin{enumerate}
\item[${\sf L1}$]\label{ilax:l1}
        $\Box(A\rightarrow B)\rightarrow(\Box A\rightarrow\Box B)$
\item[${\sf L2}$]\label{ilax:l2}
        $\Box A\rightarrow \Box\Box A$
\item[${\sf L3}$]\label{ilax:l3}
        $\Box(\Box A\rightarrow A)\rightarrow\Box A$
\item[${\sf J1}$]\label{ilax:j1}
        $\Box(A\rightarrow B)\rightarrow A\rhd B$
\item[${\sf J2}$]\label{ilax:j2}
        $(A\rhd B)\wedge (B\rhd C)\rightarrow A\rhd C$
\item[${\sf J3}$]\label{ilax:j3}
        $(A\rhd C)\wedge (B\rhd C)\rightarrow A\vee B\rhd C$
\item[${\sf J4}$]\label{ilax:j4}
        $A\rhd B\rightarrow(\Diamond A\rightarrow \Diamond B)$
\item[${\sf J5}$]\label{ilax:j5}
        $\Diamond A\rhd A$
\end{enumerate}
\end{definition}
We will write $\il \vdash \varphi$ for $\varphi \in \il$. An \il-derivation
or \il-proof of $\varphi$ is a finite sequence of formulae ending on 
$\varphi$, each being a logical tautology, an instantiation of one
of the axiom schemata of \il, or the result of applying either Modus Ponens or
Necessitation to formulas earlier in the sequence.
Clearly, $\il \vdash \varphi$ iff there is an \il-proof of $\varphi$.

Sometimes we will write $\il \vdash \varphi \rightarrow \psi \rightarrow \chi$
as short for 
$\il \vdash \varphi \rightarrow \psi \ \& \ \il \vdash \psi \rightarrow \chi$.
Similarly for $\rhd$. We adhere to a similar convention when we employ 
binary relations. Thus, $xRyS_xz\Vdash B$ is short for
$xRy \ \& \ yS_xz \ \& \ z \Vdash B$, and so on.

Sometimes we will consider the part of \il that does not contain the 
$\rhd$-modality. This is the well-known  provability logic \gl, whose
axiom schemata are ${\sf L1}$-${\sf L3}$. The axiom schema ${\sf L3}$ is
often referred to as L\"ob's axiom.

\begin{lemma}\label{lemm:basicil} 
\begin{enumerate}
\item \label{lemm:basicil:box}
$\il \vdash \Box A \leftrightarrow \neg A \rhd \bot$
\item
$\il \vdash A \rhd A \wedge \Box \neg A$
\item
$\il \vdash A \vee \Diamond A \rhd A$
\end{enumerate}
\end{lemma}

\begin{proof}
All of these statements have very easy proofs. We give an informal proof of the 
second statement. Reason in \il. It is easy to
see $A \rhd (A\wedge \Box \neg A) \vee (A\wedge \Diamond A)$. By ${\sf L3}$
we get $\Diamond A \rightarrow \Diamond (A \wedge \Box \neg A)$. Thus,
$A\wedge \Diamond A \rhd \Diamond (A \wedge \Box \neg A)$ and by 
${\sf J5}$ we get $\Diamond (A \wedge \Box \neg A) \rhd A \wedge \Box \neg A$.
As certainly $A\wedge \Box \neg A\rhd A\wedge \Box \neg A$ we have that
$(A\wedge \Box \neg A) \vee (A \wedge \Diamond A)\rhd  A\wedge \Box \neg A$ and
the result follows from transitivity of $\rhd$.
\end{proof}

Apart from the axiom schemata exposed in Definition \ref{defi:il} we will
on occassion consider other axiom schemata too.

\begin{enumerate}
\item[${\sf M}$] 
$A \rhd B \rightarrow A \wedge \Box C \rhd B \wedge \Box C$

\item[${\sf P}$]
$A \rhd B \rightarrow \Box (A \rhd B)$

\item[${\sf M_0}$]
$A \rhd B \rightarrow \Diamond A \wedge \Box C \rhd B \wedge \Box C$

\item[${\sf W}$]
$A \rhd B \rightarrow A \rhd B \wedge \Box \neg A$

\item[${\sf W^*}$]
$A \rhd B \rightarrow B \wedge \Box C \rhd B \wedge \Box C \wedge \Box \neg A$

\item[${\sf P_0}$]
$A \rhd \Diamond B \rightarrow \Box (A \rhd B)$

\item[${\sf R}$]
$A\rhd B \rightarrow \neg (A \rhd \neg C) \rhd B \wedge \Box C$

\end{enumerate}

If $\sf X$ is a set of axiom schemata we will denote by \extil{X} the
logic that arises by adding the axiom schemata in $\sf X$ to \il.
Thus, \extil{X} is the smallest set of formulas being closed under the
rules of Modus Ponens and Necessitation and containing all tautologies
and all instantiations of the axiom schemata 
of \il (${\sf L1}$-${\sf J5}$) and of the axiom schemata of $\sf X$.
Instead of writing \extil{\{ M_0,W\}} we will write \extil{M_0W} and so on.

We write $\extil{X} \vdash \varphi$ for $\varphi \in \extil{X}$. An 
\extil{X}-derivation or \extil{X}-proof of $\varphi$ is a finite sequence 
of formulae ending on 
$\varphi$, each being a logical tautology, an instantiation of one
of the axiom schemata of \extil{X}, 
or the result of applying either Modus Ponens or
Necessitation to formulas earlier in the sequence.
Again, $\extil{X} \vdash \varphi$ iff there is an \extil{X}-proof
of $\varphi$. For a schema $\sf Y$, we write $\extil{X} \vdash {\sf Y}$
if \extil{X} proves every instantiation of $\sf Y$.

\begin{definition}
Let $\Gamma$ be a set of formulas. We say that $\varphi$ is
provable from $\Gamma$ in \extil{X} and write $\Gamma \vdash_{\extil{X}} \varphi$,
iff there is a finite sequence 
of formulae ending on 
$\varphi$, each being a theorem of \extil{X}, a formula from 
$\Gamma$, or the result of applying
Modus Ponens to formulas earlier in the sequence.
\end{definition}
Clearly we have 
$\varnothing \vdash_{\extil{X}} \varphi \Leftrightarrow \extil{X}\vdash \varphi$.
In the sequel we will often write just $\Gamma \vdash \varphi$
instead of $\Gamma \vdash_{\extil{X}} \varphi$ if the context allows us so.
It is well known that we have a deduction theorem for this notion of
derivability.

\begin{lemma}[Deduction theorem]
$\Gamma , A \vdash_{\extil{X}} B \Leftrightarrow \Gamma \vdash_{\extil{X}} A \rightarrow B$
\end{lemma}

\begin{proof}
``$\Leftarrow$'' is obvious
and
%. For, let 
%$\sigma, A \rightarrow B$ be an \extil{X}-proof of $A\rightarrow B$ from
%$\Gamma$. Then $\sigma, A \rightarrow B, A , B$ is an \extil{X}-proof
%of $B$ from $\Gamma , A$.
%
%
%
%
``$\Rightarrow$'' goes by induction on the length $n$ of the \extil{X}-proof
$\sigma$ of $B$ from $\Gamma , A$. 
%If $n{=}1$, then 
%$\sigma = B$ and $B \in \Gamma \cup \{ A\}$. If 
%$B=A$, clearly $\Gamma \vdash_{\extil{X}} A \rightarrow A$. If $B \in \Gamma$, also
% $\Gamma \vdash_{\extil{X}} B$.

If $n{>}1$, then $\sigma = \tau, B$, where $B$ is obtained from some
$C$ and $C\rightarrow B$ occurring earlier in $\tau$. Thus we can find subsequences 
$\tau'$ and $\tau''$ of $\tau$ such that 
$\tau', C$ and $\tau'', C\rightarrow B$ are \extil{X}-proofs from 
$\Gamma , A$. By the induction hypothesis we find \extil{X}-proofs from $\Gamma$
of the form $\sigma' , A \rightarrow C$ and 
$\sigma'' , A \rightarrow (C \rightarrow B)$. We now use the
tautology 
$(A \rightarrow (C \rightarrow B)) \rightarrow 
((A \rightarrow C)\rightarrow(A \rightarrow B))$
to get an \extil{X}-proof of $A\rightarrow B$ from $\Gamma$. 
%Namely
%$\sigma' , A \rightarrow C , \sigma'' , A \rightarrow (C \rightarrow B),
%(A \rightarrow (C \rightarrow B)) \rightarrow 
%((A \rightarrow C)\rightarrow(A \rightarrow B)), 
%(A \rightarrow C)\rightarrow(A \rightarrow B),A \rightarrow B$.
\end{proof}

\begin{definition}
A set $\Gamma$ is \extil{X}-consistent iff $\Gamma \not \vdash_{\extil{X}} \bot$.
An \extil{X}-consistent set is maximal \extil{X}-consistent if for 
any $\varphi$, either $\varphi \in \Gamma$ or $\neg \varphi \in \Gamma$.
\end{definition}

\begin{lemma}\label{lemm:lindenbaum}
Every \extil{X}-consistent set can be extended to a maximal \extil{X}-consistent one.
\end{lemma}

\begin{proof}
This is Lindebaums lemma for \extil{X}. We can just do the regular argument
as we have the deduction theorem. Note that there are countably many different
formulas.
\end{proof}

We will often abbreviate ``maximal consistent set'' by \mcs and 
refrain from explicitly mentioning the logic \extil{X} when the
context allows us to do so.
We define three useful relations on \mcs's, the \emph{successor} relation $\sucs$, 
the \emph{$C$-critical successor} relation $\crit{C}$ and the 
\emph{Box-inclusion} relation $\boxin$.

\begin{definition}\label{defi:mcsrel}
Let $\Gamma$ and $\Delta$ denote maximal \extil{X}-consistent sets.
\begin{itemize}
\item
$\Gamma \sucs \Delta \eqbydef \Box A \in \Gamma \Rightarrow A, \Box A \in \Delta$
\item
$\Gamma \crit{C} \Delta \eqbydef A \rhd C \in \Gamma 
\Rightarrow \neg A , \Box \neg A \in \Delta$
\item
$\Gamma  \boxin \Delta \eqbydef \Box A \in \Gamma \Rightarrow \Box A \in \Delta$
\end{itemize}
\end{definition}

It is clear that $\Gamma \crit{C} \Delta \Rightarrow \Gamma \sucs \Delta$.
For, if $\Box A \in \Gamma$ then $\neg A \rhd \bot \in \Gamma$. Also 
$\bot \rhd C \in \Gamma$, whence $\neg A \rhd C \in \Gamma$. If now 
$\Gamma \crit{C} \Delta$ then $A , \Box A \in \Delta$, whence 
$\Gamma \sucs \Delta$. It is also clear that 
$\Gamma \crit{C} \Delta \sucs \Delta' \Rightarrow \Gamma \crit{C} \Delta'$.

\begin{lemma}\label{lemm:botcrit}
Let $\Gamma$ and $\Delta$ denote maximal \extil{X}-consistent sets. We have
$\Gamma \sucs \Delta$ iff $\Gamma \crit{\bot} \Delta$.
\end{lemma}

\begin{proof}
Above we have seen that $\Gamma \crit{A} \Delta \Rightarrow \Gamma \sucs \Delta$.
For the other direction suppose now that $\Gamma \sucs \Delta$. 
If $A \rhd \bot \in \Gamma$ then, by Lemma \ref{lemm:basicil}.\ref{lemm:basicil:box}, 
$\Box \neg A \in \Gamma$ whence $\neg A, \Box \neg A \in \Delta$.
\end{proof}
\subsection{Semantics}

Interpretability logics come with a Kripke-like semantics.
As the signature of our language is countable, we shall only consider 
countable models.

\begin{definition}\label{defi:frames}
An \il-frame is a triple $\langle W,R,S \rangle$. Here $W$ is a non-empty countable
universe,
$R$ is a binary relation on $W$ and $S$ is a set of binary relations on $W$, indexed 
by elements of $W$.
The $R$ and $S$ satisfy the following requirements.
\begin{enumerate}
\item
$R$ is conversely well-founded\footnote{A relation $R$ on $W$ is called 
conversely well-founded if every non-empty subset of $W$ has an $R$-maximal element.}

\item
$xRy \ \& \ yRz \rightarrow xRz$

\item
$yS_xz \rightarrow xRy \ \& \ xRz$

\item
$xRy \rightarrow yS_x y$

\item
$xRyRz \rightarrow yS_x z$ \label{defi:point:inclusion}

\item
$uS_x v S_x w \rightarrow u S_x w$

\end{enumerate}
\end{definition}
\il-frames are sometimes also called Veltman frames.
We will on occasion speak of $R$ or $S_x$ transitions instead of relations.
If we write $ySz$, we shall mean that $yS_xz$ for some $x$. 
$W$ is sometimes called the universe, or domain, of the frame and its elements
are referred to as worlds or nodes. With $x{\upharpoonright}$ we shall denote
the set $\{ y \in W \mid x Ry \}$.
We will often represent $S$ by a ternary relation in the canonical way, writing
$\langle x,y,z\rangle$ for $yS_xz$.

\begin{definition}
An \il-model is a quadruple 
$\langle W, R , S, \Vdash \rangle$. Here $\langle W, R , S, \rangle$ is 
an \il-frame and $\Vdash$ is a subset of $W \times \prop$. We write
$w \Vdash p$ for $\langle w,p\rangle \in \ \Vdash$.
As usual, $\Vdash$ is extended to a subset $\widetilde{\Vdash}$ of 
$W \times \formil$ by demanding the following.
\begin{itemize}

\item
$w \widetilde{\Vdash} p$ iff $w \Vdash p$ for $p \in \prop$

\item
$w \not \widetilde{\Vdash} \bot$

\item
$w \widetilde{\Vdash} A \rightarrow B$ iff $w \not \widetilde{\Vdash} A$ or 
$w \widetilde{\Vdash} B$

\item
$w \widetilde{\Vdash} \Box A$ iff 
$\forall v \ (wRv \Rightarrow v \widetilde{\Vdash} A)$

\item
$w \widetilde{\Vdash} A \rhd B$ iff 
$\forall u \ (w R u \wedge u\widetilde{\Vdash} A
\Rightarrow \exists v \ (u S_w v  \widetilde{\Vdash} B))$

\end{itemize}
\end{definition}

Note that $\widetilde{\Vdash}$ is completely determined by $\Vdash$.
Thus we will denote $\widetilde{\Vdash}$ also by $\Vdash$. We call
$\Vdash$ a forcing relation. The $\Vdash$-relation depends on the model
$M$. If 
%there is chance of confusion, 
necessary, we will write 
$M,w \Vdash \varphi$, if not, we will just write 
$w \Vdash \varphi$. In this case we say that $\varphi$ holds at
$w$, or that $\varphi$ is forced at $w$.
We say that \emph{ $p$ is in the range of $\Vdash$} if $w\Vdash p$ for 
some $w$.

If $F= \langle W,R,S\rangle$ is an \il-frame, we will write 
$x \in F$ to denote $x\in W$ and similarly for \il-models.
Attributes on $F$ will be inherited by its constituent parts.
For example $F_i = \langle W_i,R_i,S_i\rangle$. 
Often however we will write $F_i \models xRy$ instead
of $F_i\models x R_i y$ and likewise for the $S$-relation.
This notation is consistent with notation in first order logic
where the symbol $R$ is interpreted in the structure $F_i$ as
$R_i$.

If 
$M=\langle W,R,S, \Vdash \rangle$, we say that $M$ is based on
the frame $\langle W,R,S\rangle$ and we call $\langle W,R,S\rangle$
its underlying frame.

If $\Gamma$ is a set of formulas, we will write $M,x\Vdash \Gamma$ as 
short for $\forall \, \gamma {\in }\Gamma \ M,x \Vdash \gamma$. We have
similar reading conventions for frames and for validity.

\begin{definition}[Generated Submodel]
Let $M = \langle W, R, S , \Vdash \rangle$ be an \il-model
and let $m\in M$. We define $m{\upharpoonright}*$ to be the
set $\{ x\in W \mid x{=}m \vee mRx \}$. 
By $M{\upharpoonright}m$ we denote the 
submodel generated by $m$ defined as follows.
\[
M{\upharpoonright}m \eqbydef \langle m{\upharpoonright}*, 
R\cap (m{\upharpoonright}*)^2 , \bigcup_{x\in m{\upharpoonright}*}
S_x \cap (m{\upharpoonright}*)^2, \Vdash \cap 
(m{\upharpoonright}* \times {\sf Prop})  \rangle
\]
\end{definition}

\begin{lemma}[Generated Submodel Lemma]\label{lemm:gesulem}
Let $M$ be an \il-model and let $m\in M$. For all 
formulas $\varphi$ and all $x \in m{\upharpoonright}*$ we have that
\[
M{\upharpoonright}m,x\Vdash \varphi \ \ \mbox{ iff }\ \ 
M,x\Vdash \varphi .
\]
\end{lemma}

\begin{proof}
By an easy induction on the complexity of $\varphi$.
%We will only comment on one direction of the case 
%$\varphi = A\rhd B$. So, we suppose that for some 
%$x \in m{\upharpoonright}*$ we have
%$M,x \Vdash A \rhd B$, and will show that
%$M{\upharpoonright}m,x \Vdash A \rhd B$. If now
%$M{\upharpoonright}m,y \Vdash A$ with 
%$M{\upharpoonright}m \models xRy$, then also
%$M\models xRy$, whence, by the induction hypothesis,
%$M,y\models A$. We can thus find a $z$ with 
%$M\models yS_xz$ and $M,z\Vdash B$. As
%$x\in m{\upharpoonright}*$, we see that
%$M{\upharpoonright}m\models yS_xz$. By the induction hypothesis
%$M{\upharpoonright}m,z\Vdash B$.
\end{proof}

\medskip

We say that an \il-model makes a formula $\varphi$ true, and write
$M\models \varphi$, if $\varphi$ is forced in all the nodes
of $M$. In a formula we write
\[
M\models \varphi \equivbydef \forall \,  w {\in} M \ w\Vdash \varphi. 
\]
If  $F= \langle W,R,S\rangle$ is an \il-frame and $\Vdash$ a
subset of $W \times \prop$, we denote by $\langle W, \Vdash \rangle$
the \il-model that is based on $F$ and has forcing relation
$\Vdash$. We say that a frame $F$ makes a formula $\varphi$ true, and
write $F\models \varphi$, if any model based on $F$ makes $\varphi$ true.
In a second-order formula:
\[
F \models \varphi \equivbydef \forall \Vdash \ 
\langle F , \Vdash \rangle \models \varphi
\]

We say that an \il-model or frame makes a scheme true if it makes all 
its instantiations true. If we want to express this by a formula we should have
a means to quantify over all instantiations. For example, we could 
regard an instantiation of a scheme $\sf X$ as a substitution $\sigma$
carried out on $\sf X$ resulting in ${\sf X}^{\sigma}$. We do not 
wish to be very precise here, as it is clear what is meant. Our definitions thus
read
\[
F \models {\sf X} \mbox{ iff } \forall \sigma \ F \models {\sf X}^{\sigma}
\]
for frames $F$, and
\[
M \models {\sf X} \mbox{ iff } \forall \sigma \ M \models {\sf X}^{\sigma}
\]
for models $M$. 
Sometimes we will also write $F\models \extil{X}$ for $F\models {\sf X}$.

It turns out that checking the validity of a scheme on a frame is fairly easy.
If $\sf X$ is some scheme\footnote{Or a set of schemata. All of 
our reasoning generalizes without problems to sets of schemata. We will therefore no
longer mention the distinction.}, let $\tau$ be some base substitution that sends
different placeholders to different propositional variables.

\begin{lemma}
Let $\sf X$ be a scheme, and $\tau$ be a corresponding base substitution as
described above. Let $F$ be an \il-frame. We have
\[
F \models {\sf X}^{\tau} \Leftrightarrow \forall \sigma \ 
F \models {\sf X}^{\sigma}.
\]
\end{lemma}

%\begin{absorb}
%Nog even nalopen dit bewijs!
%\end{absorb}

\begin{proof}
If $\forall \sigma \ 
F \models {\sf X}^{\sigma}$, then certainly 
$F \models {\sf X}^{\tau}$, thus we should concentrate on the other direction.
Thus, assuming $F\models {\sf X}^{\tau}$ we fix some
$\sigma$ and $\Vdash$ and set out to prove 
$\langle F , \Vdash \rangle \models {\sf X}^{\sigma}$.
We define another forcing relation $\Vdash'$ on $F$ by saying that for 
any place holder $A$ in $\sf X$ we have
\[
w \Vdash' \tau (A) \equivbydef \langle F , \Vdash \rangle \models \sigma (A)
\]
By induction on the complexity of a subscheme\footnote{It is clear what
this notion should be.} $\sf Y$ of $\sf X$ we can now prove
\[
\langle F, \Vdash' \rangle , w \Vdash' {\sf Y}^{\tau} \Leftrightarrow
\langle F, \Vdash \rangle , w \Vdash {\sf Y}^{\sigma} .
\]
By our assumption we get that 
$\langle F, \Vdash \rangle , w \Vdash {\sf X}^{\sigma}$.
\end{proof}

If $\chi$ is some formula in first, or higher, order predicate logic, 
we will evaluate 
$F \models \chi$ in the standard way. In this case $F$ is considered as a
structure of first or higher order predicate logic. We will not be too formal 
about these matters as the context will always dict us which reading to choose.

\begin{definition}\label{defi:framecondition}
Let ${\sf X}$ be a scheme of interpretability logic. We say that a 
formula $\mathcal{C}$ in first or higher order predicate logic is
a frame condition of $\sf X$ if
\[
F\models \mathcal{C} \ \ \mbox{ iff } F \models {\sf X}.
\]
\end{definition} 

The $\mathcal{C}$ in Definition \ref{defi:framecondition} is also called 
the frame condition of the logic \extil{X}.
A frame satisfying the \extil{X} frame condition is often called
an \extil{X}-frame. In case no such frame condition exists, an 
\extil{X}-frame resp.\ model is just a frame resp.\ model, 
validating $\sf X$.

The semantics for interpretability logics is good in the 
sense that we have the necessary soundness results.

\begin{lemma}[Soundness]\label{lemm:ilsound}
$\il \vdash \varphi \Rightarrow \forall F \  F \models \varphi$
\end{lemma}

\begin{proof}
By induction on the length of an \il-proof of $\varphi$.
The requirements on $R$ and $S$ in Definition \ref{defi:frames} are precisely 
such that the axiom schemata hold. Note that all axiom schemata have their 
semantical counterpart except for the schema 
$(A \rhd C)\wedge(B \rhd C) \rightarrow A \vee B \rhd C$.
\end{proof}

\begin{lemma}[Soundness]\label{lemm:ilxsound}
Let $\mathcal{C}$ be the frame condition of the logic \extil{X}. We have
that
\[
\extil{X} \vdash \varphi \Rightarrow \forall F \ 
(F \models \mathcal{C} \Rightarrow F \models \varphi).
\]
\end{lemma}

\begin{proof}
As that of Lemma \ref{lemm:ilsound}, plugging in the definition of the 
frame condition at the right places. Note that we only need the
direction $F\models \mathcal{C} \Rightarrow F \models X$ in the proof.
\end{proof}

\begin{corollary}\label{coro:maximalsets}
Let $M$ be a model satisfying the \extil{X} frame condition, and let 
$m\in M$. We have that 
$\Gamma \eqbydef \{  \varphi \mid M,m\Vdash \varphi \}$
is a maximal \extil{X}-consistent set.
\end{corollary}

\begin{proof}
Clearly $\bot \notin \Gamma$. Also $A\in \Gamma$ or $\neg A \in \Gamma$.
By the soundness lemma, Lemma \ref{lemm:ilxsound}, we see that
$\Gamma$ is closed under \extil{X} consequences.
\end{proof}

%\begin{absorb}
\begin{lemma}\label{lemm:modelsound}
Let $M$ be a model such that $\forall\,  w{\in}M \ $ 
$w \Vdash \extil{X}$ then $\extil{X} \vdash \varphi \Rightarrow M\models \varphi$.
\end{lemma}

\begin{proof}
By induction on the derivation of $\varphi$.
\end{proof}
%Relate a completeness result for $\vdash \varphi$ iff $\models \varphi$
%to $\Gamma \vdash \varphi$ iff $\Gamma \Vdash \varphi$ by means of the 
%deduction theorem.
%\end{absorb}

A modal logic \extil{X} with frame condition $\mathcal{C}$ is called 
complete if we have the implication the other way round too. That is,
\[
\forall F \ 
(F \models \mathcal{C} \Rightarrow F \models \varphi) \Rightarrow \extil{X} 
\vdash \varphi .
\]

A major concern of this paper is the question whether a given modal 
logic \extil{X} is complete.

%\begin{absorb}
%Hier een tabel met een overzicht van (on)volledigheidsresultaten?
%\end{absorb}

\begin{definition}
$\Gamma \Vdash_{\extil{X}} \varphi$ iff $\forall M \ M\models \extil{X} \Rightarrow 
(\forall \, m{\in}M \ 
[M,m\Vdash \Gamma \Rightarrow M,m \Vdash \varphi])$
\end{definition}

\begin{lemma}
Let $\Gamma$ be a finite set of formulas and let \extil{X} be a complete
logic. We have that 
$\Gamma \vdash_{\extil{X}} \varphi$ iff $\Gamma \Vdash_{\extil{X}} \varphi$.
\end{lemma}

\begin{proof}
Trivial. By the deduction theorem
$\Gamma \vdash_{\extil{X}}\varphi \Leftrightarrow \vdash_{\extil{X}} 
\bigwedge \Gamma \rightarrow 
\varphi$. By our assumption on completeness we get the result. Note that the 
requirement that $\Gamma$ be finite is necessary, as our modal logics are in
general not compact (see also Section \ref{subs:mainingredients}).
\end{proof}

Often we shall need to compare different frames or models.
If $F = \langle W,R,S \rangle$ and $F' = \langle W',R',S' \rangle$ are
frames, we say that $F$ is a subframe of $F'$ and write $F\subseteq F'$,
if $W \subseteq W'$, $R \subseteq R'$ and $S \subseteq S'$. Here 
$S \subseteq S'$ is short for $\forall \,  w{\in} W \ (S_w \subseteq S_w')$.

\subsection{Arithmetic}
%
%\begin{absorb}
%\begin{itemize}
%\item
%Arithmetical realizations
%
%\item
%$\Sigma_1 !$-sentences%
%
%\item
%$\Sigma_1$-sentences, arithmetic hierarchy, $\Sigma_n(T)$.%
%
%\item
%Reflexive theory%
%
%\item
%Essentially reflexive theory, \ilm%
%
%\item
%Finitely axiomatizable theories, \ilp%
%
%\item
%T%he logic of all reasonable arithmetical theories
%
%\item
As with (almost) all interesting occurrences of modal logic, 
interpretability logics are used to study a hard mathematical notion.
Interpretability logics, as their name slightly suggests, are 
used to study the notion of formal interpretability. In this subsection
we shall very briefly say what this notion is and how modal logic is 
used to study it.

We are interested in first order theories in the language of arithmetic. All 
theories we will consider will thus be arithmetical theories. 
Moreover, we want our theories to have a certain minimal strength.
That is, they should contain a certain core theory, say 
${\sf I}\Delta_{0} + \Omega_1$ from \cite{HP}. This will allow
us to do reasonable coding of syntax. We call these theories reasonable
arithmetical theories.

Once we can code syntax, we can write down a decidable predicate 
${\sf Proof}_T(p,\varphi)$ that holds on the standard model
precisely when $p$ is a $T$-proof of $\varphi$.\footnote{We take
the liberty to not make a distinction between a syntactical object
and its code.} We get a provability predicate by quantifying
existentially, that is, 
${\sf Prov}_T (\varphi ) \eqbydef \exists p \ {\sf Proof}_T (p, \varphi )$.

We can use these coding techniques to code the notion of
formal interpretability too.
Roughly, a theory $U$ interprets a theory $V$ if there is some
sort of translation so that every theorem of $V$ is under that
translation also a theorem of $U$.

\begin{definition}
Let $U$ and $V$ be reasonable arithmetical theories.
An interpretation $j$ from $V$ in $U$ is a pair
$\langle \delta , F \rangle$. Here, $\delta$ is called a 
domain specifier. It is a formula with one free variable.
The $F$ is a map that sends an $n$-ary relation symbol of
$V$ to a formula of $U$ with $n$ free variables. 
(We treat functions and constants as relations with additional properties.)
The interpretation $j$ induces a translation from formulas 
$\varphi$ of $V$ to
formulas $\varphi^j$ of $U$ by replacing relation symbols by their
corresponding formulas and by relativizing quantifiers to $\delta$.
We have the following requirements.
\begin{itemize}
\item
$(R (\vec{x}))^j = F(R) (\vec{x})$

\item
The translation induced by $j$ commutes with the boolean 
connectives. Thus, for example,
$(\varphi \vee \psi)^j = \varphi^j \vee \psi^j$.
In particular
$(\bot)^j=(\vee_{\varnothing})^j = \vee_{\varnothing}=\bot$

\item
$(\forall x \ \varphi)^j= \forall x \ (\delta (x) \rightarrow \varphi^j)$

\item
$V\vdash \varphi \Rightarrow U\vdash \varphi^j$

\end{itemize}
We say that $V$ is interpretable in $U$ if there exists an interpretation
$j$ of $V$ in $U$.
\end{definition}

Using the ${\sf Prov}_T (\varphi)$ predicate, it is possible
to code the notion of formal interpretability in arithmetical
theories. This gives rise to a formula ${\sf Int}_T (\varphi , \psi )$,
to hold on the standard model precisely when $T+ \psi$
is interpretable in $T+\varphi$. This formula is related to the
modal part by means of arithmetical realizations.

\begin{definition}
An arithmetical realization $*$ is a mapping that assigns
to each propositional variable an arithmetical sentence. This mapping
is extended to all modal formulas in the following way.
\begin{itemize}
\item[-]
$(\varphi \vee \psi)^* = \varphi^* \vee \psi^*$ and likewise for 
other boolean connectives. In particular
$\bot^* = (\vee_{\varnothing})^*=\vee_{\varnothing}=\bot$.

\item[-]
$(\Box \varphi)^*= {\sf Prov}_T (\varphi^*)$

\item[-]
$(\varphi \rhd \psi)^*= {Int}_T (\varphi^*,\psi^*)$
\end{itemize}
\end{definition}
\noindent
From now on, the $*$ will always range over realizations.
Often we will write $\Box_T \varphi$ instead of ${\sf Prov}_T (\varphi)$
or just even $\Box \varphi$. The $\Box$ can thus denote both a 
modal symbol and an arithmetical formula. For the $\rhd$-modality we
adopt a similar convention. We are confident that no
confusion will arise from this.

\begin{definition}
An interpretability principle of a theory $T$ is a modal formula
$\varphi$ that is provable in $T$ under any realization. That is,
$\forall *\ T\vdash \varphi^*$. The interpretability logic of a
theory $T$, we write \intl{T}, is the set of
all interpretability principles.
\end{definition}

Likewise, we can talk of the set of all provability principles of
a theory $T$, denoted by
\provl{T}. Since the famous result by Solovay, \provl{T} is known for a large 
class of theories $T$.

\begin{theorem}[Solovay \cite{Sol76}]
$\provl{T} = \gl$ for any reasonable arithmetical theory $T$. 
\end{theorem}

For two classes of theories, \intl{T} is known.

\begin{definition}
A theory $T$ is reflexive if it proves the consistency of any of 
its finite subtheories. It is essentially reflexive if any finite
extension of it is reflexive.
\end{definition}

\begin{theorem}[Berarducci \cite{bera:inte90}, Shavrukov \cite{shav:logi88}]\label{theo:shav}
If $T$ is an essentially reflexive theory, then $\intl{T}=\ilm$.
\end{theorem}

\begin{theorem}[Visser \cite{viss:inte90}]
If $T$ is finitely axiomatizable, then $\intl{T}=\ilp$.
\end{theorem}

\begin{definition}\label{defi:ilall}
The interpretability logic of all reasonable arithmetical theories, 
we write
\intl{All}, is the set of formulas $\varphi$ such that
$\forall T  \, \forall *\ T\vdash \varphi^*$. Here the $T$ ranges over
all the reasonable arithmetical theories.
\end{definition}

For sure \intl{All} should be in the intersection of \ilm and \ilp.
Up to now, \intl{All} is unknown. In \cite{joo:prol00} it is conjectured
to be \extil{P_0W^*}. It is one of the major open problems in the 
field of interpretability logics, to characterize \intl{All} in a modal
way.

We conclude this subsection with a definition of the arithmetical 
hierarchy. This definition is needed in Section \ref{sect:essigma}.
\begin{definition}Inductively the following classes of arithmetical formulae are
defined.
\begin{itemize}

\item
Arithmetical formulas with only bounded quantifiers in it
are called $\Delta_0$, $\Sigma_0$ or $\Pi_0$-formulas.
\item
If $\varphi$ is a $\Pi_n$ or $\Sigma_{n+1}$-formula, then 
$\exists x \ \varphi$ is a $\Sigma_{n+1}$-formula.
\item
If $\varphi$ is a $\Sigma_n$ or $\Pi_{n+1}$-formula, then 
$\forall x \ \varphi$ is a $\Pi_{n+1}$-formula.
\end{itemize}
\end{definition}

\begin{definition}
Let $\varphi$ be an arithmetical formula.
\begin{itemize}
\item[-]
$\varphi \in \Pi_n(T) \ \mbox{ iff } \ \exists \, \pi {\in} \Pi_n \ 
T\vdash \varphi \leftrightarrow \pi$
\item[-]
$\varphi \in \Sigma_n(T) \ \mbox{ iff } \ \exists \, \sigma {\in} \Sigma_n \ 
T\vdash \varphi \leftrightarrow \sigma$
\item[-]
$\varphi \in \Delta_n(T) \ \mbox{ iff } \ \exists \, \pi {\in} \Pi_n \ \& \ 
\exists \, \sigma {\in} \Sigma_n \  T\vdash (\varphi \leftrightarrow \pi)
\wedge (\varphi \leftrightarrow \sigma)$
\end{itemize}

\end{definition}

Sometimes, if no confusion can arise,  we will write $\Sigma_n !$-formulas instead of
$\Sigma_n$-formulas and $\Sigma_n$-formulas instead of $\Sigma_n (T)$-formulas.

\section{General exposition of the construction method}\label{sect:overview}
\newcommand{\newforce}{{\ensuremath{\| \! \! \! \sim}}\xspace}
\newcommand{\notnewforce}{{\ensuremath{\| \! \! \! \not \sim}}\xspace}

A central result in this paper is given by a construction method that shall be worked out in the next section.
Most of the applications of this construction method
deal with modal completeness of a certain logic \extil{X}. 
More precisely, showing that a logic \extil{X} is modally 
complete amounts to constructing, or finding,
whenever $\extil{X} \not \vdash \varphi$, a model $M$ of \extil{X} and an $x\in M$
such that
$M,x\Vdash \neg \varphi$.
We will employ our construction method for this particular 
model construction.

In this section, we shall lay out the basic ideas which are involved in the construction method. In particular, we will not always give precise definitions of the notions we work with.
All the definitions can be found in Section \ref{sect:motor}.
 
\subsection{The main ingredients of the construction method}\label{subs:mainingredients}
%\subsection{The main concepts and ideas in the construction method}
%\begin{absorb}
%This subsection gives a motivation for the definitions in Subsection \ref{defins}
%and can be skipped by readers already familiar with completeness proofs for
%interpretability logics.
As we mentioned above, a modal completeness proof of a logic 
\extil{X} amounts to a uniform model construction to obtain
$M,x \Vdash \neg \varphi$ for $\extil{X} \not \vdash \varphi$.
If $\extil{X} \not \vdash \varphi$, then $\{ \neg \varphi \}$
is an \extil{X}-consistent set and thus, by a version of
Lindenbaum's Lemma (Lemma \ref{lemm:lindenbaum}), it is
extendible to a maximal \extil{X}-consistent set. On the
other hand, once we have an \extil{X}-model 
$M,x\Vdash \neg \varphi$, we can find, by Corollary \ref{coro:maximalsets}
a maximal \extil{X}-consistent set $\Gamma$ with $\neg \varphi \in \Gamma$.
This $\Gamma$ can simply be defined as the set of all formulas that hold 
at $x$.

To go from a maximal \extil{X}-consistent set to a model is always 
the hard part. This part is carried out in our construction method.
In this method, the maximal consistent set is somehow partly unfolded 
to a model.

Often in these sort of model constructions, the worlds in the model are
\mcs's. For propositional variables one then defines
$x\Vdash p \ \mbox{ iff.} \ p\in x$. In the setting of interpretability logics it is
%often necessary 
sometimes inevitable to use the same \mcs in different places in the 
model.\footnote{As the truth definition of $A\rhd B$ has a $\forall \exists$ 
character, the corresponding notion of bisimulation is rather involved.
As a consequence there is in general no obvious notion of a minimal bisimular model, 
contrary to the case of provability logics. This causes the necessity of several 
occurrences of \mcs's.} Therefore we find it convenient not to identify a world $x$
with a \mcs, but rather label it with a \mcs $\nu (x)$.
However, we will still write sometimes 
$\varphi \in x$ instead of $\varphi \in \nu (x)$.
\medskip

One complication in unfolding a \mcs to a model lies in the 
incompactness of the modal logics we consider. This, in turn, is
due to the fact that some frame conditions are not expressible in first 
order logic. As an example we can consider the following 
set.\footnote{This example comes from Fine and Rautenberg and is 
treated in Chapter 7 of \cite{Boo93}. }
\[
\Gamma \eqbydef \{ \Diamond p_0\} \cup 
\{  \Box (p_i \rightarrow \Diamond p_{i+1})\mid i\in \omega\}
\]
Clearly, $\Gamma$ is a \gl-consistent set, and any finite part of it
is satisfiable in some world in some model. However, it is 
not hard to see that in no \il-model all of $\Gamma$ can hold 
simultaneously in some world in it.

If $M$ is an \extil{X}-model and $x\in M$, then 
$\{ \varphi \mid M,x \Vdash \varphi\}$ is a \mcs.
By definition (and abuse of notation) we see that
\[
\forall x \; [x \Vdash \varphi \ \ \mbox{ iff. } \ \ \varphi \in x ] .
\]
We call this equivalence a truth lemma. 
(See for example Definition \ref{defi:truthlemma} for a more
precise formulation.) In all completeness proofs a model is
defined or constructed in which some form of a truth lemma 
holds. Now, by the observed incompactness phenomenon, we can
not expect that for every \mcs, say $\Gamma$, we can find a model
``containing'' $\Gamma$ for which a truth lemma holds in full 
generality. There are various ways to circumvent this 
complication. Often one considers
truncated parts of maximal consistent sets which are finite.
In choosing how to truncate, one is driven by two opposite forces.

On the one hand this truncated part should be small. It should be at least 
finite so that the incompactness phenomenon is blocked. The finiteness is
also a desideratum if one is interested in the decidability of a logic.

On the other hand, the truncated part should be large. It should be 
large enough to admit inductive reasoning to prove a 
truth lemma. For this, often closure under
subformulas and single negation suffices. Also, the truncated part should be
large enough so that \mcs's contain enough information to do the 
required calculation. For this, being closed under subformulas and 
single negations does not, in general, suffice. 
Examples of these sort of calculation are Lemma 
\ref{lemm:ilmdeficiencies} and Lemma \ref{lemm:ilmoexistence}.
\medskip

In our approach we take the best of both opposites. That is, we do not 
truncate at all. Like this, calculation becomes uniform, smooth and relatively
easy. However, we demand a truth lemma to hold only for finitely many formulas.

The question is now, how to unfold the \mcs containing $\neg \varphi$
to a model where $\neg \varphi$ holds in some world. We would have
such a model if a truth lemma holds w.r.t.\ a finite set \adset{D} 
containing $\neg \varphi$.

Proving that a truth lemma holds is usually done by induction on the 
complexity
of formulas. As such, this is a typical ``bottom up'' or ``inside out''
activity. On the other hand, unfolding, or reading off, the truth value of a formula
is a typical ``top down'' or ``outside in'' activity.

Yet, we do want to gradually build up a model  so that we get closer and closer to a truth
lemma. But, how could we possibly measure 
%this? 
that we come closer to a truth lemma?
Either everything is in place and 
a truth lemma holds, or 
%it does not 
%this is not so
%and then we are completely lost.
a truth lemma does not hold, in which case it seems unclear how to 
%repair 
%this.
measure to what extend it does not hold.

The gradually building up a model will take place by consecutively adding 
bits and pieces to the \mcs we started out with.
Thus somehow, we do want to measure that we come closer to a truth lemma 
by doing so. Therefore, we switch to an alternative forcing relation \newforce
that follows the ``outside in'' direction that is so characteristic to the
evaluation of $x\Vdash \varphi$, but at the same time incorporates the 
necessary elements of a truth lemma.

\[
\begin{array}{lll}
x \newforce p & \mbox{iff.} & p\in x \mbox{ \ \ \ \ \ \ \ \ \ \ \ for propositional variables $p$}\\
x \newforce \varphi \wedge \psi &\mbox{iff.} & x \newforce 
\varphi \ \& \ x\newforce \psi \mbox{ and likewise for}\\
\ &\ &\mbox{\ \ \ \ \ \ \ \ \ \ \ \ \ \ \ \ \ \ \  other boolean connectives}\\
x \newforce \varphi \rhd \psi & \mbox{iff.} & \forall y\ [ xRy \wedge \varphi \in x
\rightarrow \exists z\ (yS_xz \wedge \psi \in z)] 
\end{array}
\]
If \adset{D} is a set of sentences that is closed under subformulas and single negations,
then it is not hard to see that (see Lemma \ref{lemm:truth})
\[
\forall x \forall \, \varphi {\in} 
\adset{D} \ [ x\newforce \varphi \mbox{ iff. } \varphi \in x] \ \ \ (*)
\]
is equivalent to
\[
\forall x \forall \, \varphi {\in} 
\adset{D} \ [ x\Vdash \varphi \mbox{ iff. } \varphi \in x]. \ \ \ (**)
\]
Thus, if we want to obtain a truth lemma for a finite set \adset{D} that is closed under
single negations and subformulas, we are done if we can obtain $(*)$. But now it is clear
how we can at each step measure that we come closer to a truth lemma.
This brings us to the definition of problems and deficiencies.

A problem is some formula $\neg (\varphi \rhd \psi)\in x\cap \adset{D}$ such that
$x \notnewforce \neg (\varphi \rhd \psi)$. We define a deficiency to be a configuration
such that $\varphi \rhd \psi\in x\cap \adset{D}$ but 
$x \notnewforce \varphi \rhd \psi$. It now becomes clear how we can successively
eliminate problems and deficiencies.

A deficiency $\varphi \rhd \psi\in x\cap \adset{D}$ is a deficiency because there
is some $y$ (or maybe more of them) with $xRy$, and $\varphi \in y$, but for no $z$
with $yS_xz$, we have $\psi \in z$. This can simply be eliminated by adding a $z$
with $yS_xz$ and $\psi \in z$.

A problem $\neg (\varphi \rhd \psi)\in x\cap \adset{D}$ can be eliminated by 
adding a completely isolated $y$ to the model with $xRy$ and 
$\varphi, \neg \psi \in y$. As $y$ is completely isolated, $yS_xz \Rightarrow z=y$
and thus indeed, it is not possible to reach a world where $\psi$ holds. Now here is 
one complication.

We want that a problem or a deficiency, once eliminated, can never
re-occur. For deficiencies this complication is not so severe, as the quantifier 
complexity is $\forall \exists$. Thus, any time ``a deficiency becomes active'',
we can immediately deal with it.

With the elimination of a problem, things are 
more subtle. When we introduced $y\ni \varphi, \neg \psi$ to eliminate
a problem $\neg (\varphi \rhd \psi)\in x\cap \adset{D}$, we did indeed 
eliminate it, as for no $z$ with $yS_xz$ we have $\psi \in z$.
However, this should hold for any future expansion of the model too. Thus, any time
we eliminate a problem $\neg (\varphi \rhd \psi)\in x\cap \adset{D}$, we
introduce a world $y$ with a promise that in no future time we will be 
able to go to a world $z$ containing $\psi$ via an $S_x$-transition.
Somehow we should keep track of all these promises throughout the 
construction and make sure that all the promises are indeed kept. \
This is taken care of by our so called 
$\psi$-critical cones (see for example also \cite{DJ}). As $\psi$
is certainly not allowed to hold in $R$-successors of $y$, it 
is reasonable to demand that $\Box \neg \psi \in y$. (Where $y$ was 
introduced to eliminate the problem $\neg (\varphi \rhd \psi)\in x\cap \adset{D}$.)

Note that problems have quantifier complexity $\exists \forall$. We 
have chosen to call them problems due to their prominent existential nature.

\subsection{Some methods to obtain completeness}

For modal logics in general, quite an arsenal of methods to obtain completeness
is available. For instance the standard operations on canonical models
like path--coding (unraveling), filtrations and bulldozing (see \cite{Bla01}).
Or one can mention uniform methods like the use of Shalqvist formulas
or the David Lewis theorem \cite{Boo93}. A very secure method is
to construct counter models piece by piece. A nice example can be found
in \cite{Boo93}, Chapter 10. In \cite{HoMiVe01} and in 
\cite{hiho02} a step-by-step method is 
exposed in the setting of universal algebras. New approximations
of the model are given by moves in an (infinite) game.

For interpretability logics the available methods are rather limited in number.
In the case of the basic logic \il a relatively simple unraveling
works. Although \ilm does allow a same treatment, the proof is already
much less clear. (For both proofs, see \cite{DJ}).
However, for logics that contain \ilmo but not \ilm it is completely unclear
how to obtain completeness via an unraveling and we are forced into
more secure methods like the above mentioned building of models piece by piece.
And this is precisely what we do in this paper.

Decidability and the finite model property are two related issues that more or
less seem to divide the landscape of interpretability logics into the
same classes. That is, the proof that \il has the finite model property
is relatively easy. The same can be said about \ilm.
For logics like \ilmo the issue seems much more involved and a proper
proof of the finite model property, if one exists at all, has not been 
given yet.
Alternatively, one could resort to other methods for showing
decidability like the Mosaic method \cite{Bla01}.

\section{The construction method}\label{sect:motor}
In this section we describe our construction method in full detail. Sections \ref{sect:il}-\ref{sect:wstar} are applications of the construction method.

\subsection{Preparing the construction} \label{defins}

An $\extil{X}$-labeled frame is just a Veltman frame in which every node is labeled by a
maximal $\extil{X}$-consistent set and some $R$-transitions are labeled by a formula.
$R$-transitions labeled by a formula $C$ indicate that some $C$-criticallity is
essentially present at this place.

\begin{definition}
An \emph{$\extil{X}$-labeled frame} is a quadruple 
$\langle W, R, S, \nu \rangle$. Here $\langle W,R,S \rangle$ is an \il-frame and $\nu$
is a labeling function. The function $\nu$ assigns to each $x \in W$ a maximal 
\extil{X}-consistent set of sentences $\nu (x)$. To some pairs $ \langle x,y \rangle$
with $xRy$, $\nu$ assigns a formula $\nu (\langle x,y \rangle )$.
\end{definition}

If there is no chance of confusion we will just speak of labeled frames or even
just of frames rather than
\extil{X}-labeled frames. Labeled frames inherit all the terminology and notation
from normal frames. 
Note that an \extil{X}-labeled frame need not be, and shall in general not be,
an \extil{X}-frame. If we speak about a labeled \extil{X}-frame we 
always mean an \extil{X}-labeled \extil{X}-frame. To indicate that 
$\nu( \langle x,y \rangle ) = A$ we will sometimes write $xR^Ay$
or $\nu (x,y) = A$.

%preprint<
%\begin{absorb}
%Choose:
%The labeling $\nu$ will in general not be defined on all possible pairs. So, we do not
%deal with a function. If you want to go really formal, you can see $\nu$ as a subset 
%of $(W \cup (W \times W))\times ({\sf Form} \cup \wp ({\sf Form}))$ with some additional properties, 
%where ${\sf Form}$ is the set of all formulas.
%Alternatively one can see $\nu$ as a partial function on $(W \cup (W \times W)$ such that
%$\nu$ restricted to $W$ is total.
%\end{absorb}
%preprint>

Formally, given $F =\langle W,R,S,\nu\rangle$, one can see $\nu$
as a subset of 
$(W \cup (W \times W))\times (\formil \cup \{ \Gamma \mid \Gamma 
\mbox{ is a maximal \extil{X} consistent set} \})$ such that the following properties 
hold.
\begin{itemize}

\item[-]
$\forall \, x{\in} W \ (\langle x,y \rangle \in \nu \Rightarrow y
\mbox{ is a \mcs} )$

\item[-]
$\forall \, \langle x, y\rangle {\in} W\times W \ (\langle \langle x,y \rangle, z\rangle \in 
\nu \Rightarrow z
\mbox{ is a formula} )$

\item[-]
$\forall \, x{\in} W \exists y\ \langle x,y\rangle \in \nu$

\item[-]
$\forall x,y,y' (
%\exists y \ \langle x,y \rangle \in \nu \rightarrow 
%\forall y,y'\ (
\langle x,y\rangle \in \nu  \wedge \langle x,y'\rangle \in \nu 
\rightarrow y=y')
%)
$

\end{itemize}
We will often regard $\nu$ as a partial function on $W\cup (W\times W)$
which is total on $W$ and which has its values in
$\formil \cup \{ \Gamma \mid \Gamma \mbox{ is a maximal 
\extil{X} consistent set} \}$

\begin{remark}\label{rema:frameduality}
Every \extil{X}-labeled frame $F=\langle W, R, S, \nu \rangle$ can 
be transformed to an \il-model $\overline{F}$ in a 
uniform way by defining for propositional 
variables $p$ the valuation as 
$\overline{F},x\Vdash p$ iff.\ $p\in \nu (x)$. 
By Corollary \ref{coro:maximalsets} we can also regard any 
model 
$M$
satisfying the \extil{X} 
frame condition\footnote{We could even say, any \extil{X}-model.}
as an \extil{X}-labeled frame $\overline M$ by defining
$\nu (m) \eqbydef \{ \varphi \mid M,m \Vdash \varphi \}$.
\end{remark}

We sometimes refer to 
$\overline{F}$ as the model induced by the frame $F$. 
Alternatively we will speak about the model 
corresponding to $F$.
Note that for \extil{X}-models M, we have $\overline{\overline M}=M$,
but in general $\overline{\overline F}\neq F$ for \extil{X}-labeled frames $F$.

%
%\begin{absorb}
%Later?
%\end{absorb}
%

\begin{definition}\label{defi:critcone}
Let $x$ be a world in some \extil{X}-labeled frame $\langle W, R, S, \nu \rangle$.
The \emph{$C$-critical cone above $x$}, we write $\crone{C}{x}$, is defined inductively as
\begin{itemize}
\item
$\nu (\langle x,y \rangle)=C \Rightarrow y \in \crone{C}{x}$
\item
$x' \in \crone{C}{x} \ \& \  x'S_x y \Rightarrow y\in \crone{C}{x}$
\item
$x' \in \crone{C}{x} \ \& \  x'R y \Rightarrow y\in \crone{C}{x}$
\end{itemize}
\end{definition}

\begin{definition}\label{defi:gencone}
Let $x$ be a world in some \extil{X}-labeled frame $\langle W, R, S, \nu \rangle$.
The \emph{generalized $C$-cone above $x$}, we write $\geone{C}{x}$, is defined inductively as
\begin{itemize}
\item
$y \in \crone{C}{x} \Rightarrow  y \in \geone{C}{x}$
\item
$x' \in \geone{C}{x}  \ \& \ x'S_wz \Rightarrow z \in \geone{C}{x}$ for arbitrary $w$
\item
$x' \in \geone{C}{x}  \ \& \ x'Ry \Rightarrow y \in \geone{C}{x}$
\end{itemize}
\end{definition}

%preprint<
It follows directly from the definition that the $C$-critical cone above $x$
is part of the generalized $C$-cone above $x$. So, if $\geone{B}{x} \cap \geone{C}{x} = \varnothing$,
then certainly  $\crone{B}{x} \cap \crone{C}{x} = \varnothing$.

We also note that there is some redundancy in Definitions \ref{defi:critcone} and 
\ref{defi:gencone}. The last clause in the inductive definitions demands closure of 
the cone under $R$-successors. But from Definition 
\ref{defi:frames}.\ref{defi:point:inclusion} closure
of the cone under $R$ follows from closure of the cone under $S_x$.
We have chosen to 
explicitly  adopt the closure under the $R$. In doing so, we obtain a notion that
serves us also in the environment of so-called quasi frames (see Definition 
\ref{defi:quasiframes}) in which not necessarily 
$(x{\upharpoonright})^2\cap R \subseteq S_x$.

\begin{definition}\label{defi:truthlemma}
Let $F = \langle W, R, S, \nu \rangle$ be a labeled frame and let $\overline{F}$
be the induced \il-model. Furthermore, let \adset{D} be some set of sentences.
We say that \emph{a truth lemma holds in $F$ with respect to \adset{D}}
if $\forall \,  A {\in} \adset{D}\; \forall \, x{\in} \overline{F}$
\[
\overline{F}, x \Vdash A \Leftrightarrow A \in \nu (x).
\]
\end{definition}

If there is no chance of confusion we will omit some parameters and just say 
``a truth lemma holds at $F$'' or even ``a truth lemma holds''.
The following definitions give us a means to measure how far we are away from a
truth lemma.

\begin{definition}[Temporary definition]\footnote{We will eventually work with
Definition \ref{defi:problems}.}
Let \adset{D} be some set of sentences and let 
$F = \langle W, R, S, \nu \rangle$ be an \extil{X}-labeled frame. A \emph{\adset{D}-problem}
is a pair $\langle x, \neg (A \rhd B ) \rangle$ such that 
$\neg (A \rhd B) \in \nu (x) \cap \adset{D}$
and for every $y$ with $xRy$ we have 
$[A \in \nu (y) \Rightarrow \exists z \; (yS_xz \wedge B \in \nu (z))]$.
\end{definition}

\begin{definition}[Deficiencies]
Let \adset{D} be some set of sentences and let 
$F = \langle W, R, S, \nu \rangle$ be an \extil{X}-labeled frame. A \emph{\adset{D}-deficiency}
is a triple $\langle x,y, C \rhd D \rangle$ with $xRy$, 
$C\rhd D \in \nu (x) \cap \adset{D}$, and $C\in \nu (y)$,
but for no $z$ with $yS_xz$ we have $D\in \nu (z)$.
\end{definition}

If the set $\adset{D}$ is clear or fixed, we will just speak about problems and deficiencies.

\begin{definition}
Let $A$ be a formula. We define the \emph{single negation} of $A$, we write ${\sim} A$,
as follows. If $A$ is of the form $\neg B$ we define ${\sim} A$ to be $B$. If $A$ is
not a negated formula we set ${\sim} A := \neg A$.
\end{definition}

The next lemma shows that a truth lemma w.r.t.\ \adset{D} can be 
reformulated in the combinatoric terms of deficiencies and problems. 
(See also the equivalence of $(*)$ and $(**)$ in Section \ref{sect:overview}.)

\begin{lemma}\label{lemm:truth}
Let $F = \langle W, R, S, \nu \rangle$ be a labeled frame, and let $\adset{D}$ be 
a set of sentences closed under single negation and subformulas. A truth lemma 
holds in $F$ w.r.t.\ \adset{D} iff. there are no \adset{D}-problems nor
\adset{D}-deficiencies.
\end{lemma}

\begin{proof}
The proof is really very simple and precisely shows they interplay between all the 
ingredients.
\end{proof}

The labeled frames we will construct are always supposed to satisfy some minimal reasonable
requirements. We summarize these in the notion of adequacy.

\begin{definition}[Adequate frames]
A frame is called \emph{adequate} if the following conditions are satisfied.
\begin{enumerate}
\item
$xRy \Rightarrow \nu (x) \sucs \nu (y)$

\item
$A \neq B \Rightarrow \geone{A}{x} \cap \geone{B}{x} = \varnothing$

\item
$y \in \crone{A}{x} \Rightarrow \nu (x) \crit{A} \nu (y)$
\end{enumerate}
\end{definition}

If no confusion is possible we will just speak of frames instead of adequate
labeled frames. As a matter of fact, all the labeled frames we will see from
now on will be adequate.
In the light of adequacy it seems reasonable to work with a slightly more elegant 
definition of a \adset{D}-problem.

\begin{definition}[Problems]\label{defi:problems}
Let \adset{D} be some set of sentences. 
A \emph{\adset{D}-problem} is a pair $\langle x , \neg (A \rhd B)\rangle$ such that 
$\neg (A \rhd B) \in \nu (x) \cap \adset{D}$ 
and for no $y \in \crone{B}{x}$ we have $A \in \nu (y)$.
\end{definition}

From now on, this will be our working definition. Clearly, on adequate labeled frames,
if $\langle x , \neg (A \rhd B)\rangle$
is not a problem in the new sense, it is not a problem in the old sense.

\begin{remark}\label{rema:lemma}
It is also easy to see that the we still have the interesting half of Lemma \ref{lemm:truth}.
Thus, we still have, that a truth lemma holds if there are no deficiencies nor problems.
\end{remark}

To get a truth lemma we have to somehow get rid of problems and deficiencies.
This will be done by adding bits and pieces to the original labeled frame. Thus the 
notion of an extension comes into play.

\begin{definition}[Extension]
Let $F = \langle W,R,S, \nu \rangle$ be a labeled frame. We say that 
$F' = \langle W',R',S', \nu' \rangle$ is an \emph{extension} of $F$, we
write $F \ext F'$, if $W \subseteq W'$ and the relations in $F'$ restricted to
$F$ yield the corresponding relations in $F$. 
%Moreover we demand $W'\setminus W$ to be finite.
\end{definition}

More formally, the requirements on the restrictions in the above definition amount to  
saying that for $x,y,z \in F$ we have the 
following. 
\begin{itemize}
\item[-] $xR'y$ iff. $xRy$ 
\item[-]$yS'_xz$ iff. $yS_xz$ 
\item[-] $\nu' (x) = \nu (x)$
\item[-] $\nu' (\langle x, y \rangle)$ is defined iff.  $\nu (\langle x, y \rangle)$ 
is defined, and in this case  $\nu' (\langle x, y \rangle) = \nu (\langle x, y \rangle)$. 
\end{itemize}

A problem in $F$ is said to be \emph{eliminated} by the 
extension $F'$ if it is no longer a problem 
in $F'$. Likewise we can speak about elimination of deficiencies.

\begin{definition}[Depth]
The \emph{depth} of a finite frame $F$, we will write $\depth{F}$ is the maximal length of 
sequences of the form $x_0R \ldots R x_n$. (For convenience we define 
$\max (\varnothing)=0$.)
\end{definition}

\begin{definition}[Union of Bounded Chains]
An indexed set $\{ F_i\}_{i \in \omega}$ of labeled frames is called a \emph{chain} if
for all $i$, $F_i \ext F_{i+1}$. It is called a \emph{bounded chain} if for some 
number $n$, $\depth{F_i}\leq n$ for all $i\in \omega$. The \emph{union} of a bounded 
chain $\{ F_i\}_{i \in \omega}$ of labeled frames $F_i$ is defined as follows.
\[
\cup_{i\in \omega} F_i 
\eqbydef \langle \cup_{i\in \omega} W_i,  \cup_{i\in \omega}R_i,  \cup_{i\in \omega}S_i, 
\cup_{i\in \omega}\nu_i \rangle
\]
%F^{\infty}
\end{definition}

It is clear why we really need the boundedness condition. We want 
the union to be an \il-frame. So, certainly $R$ should be conversely 
well-founded. This can only be the case if our chain is bounded.

\subsection{The main lemma}

We now come to the main motor behind many results. It is formulated in rather general terms
so that it has a wide range of applicability. As a draw-back, we get that any application still
requires quite some work.

\begin{lemma}[Main Lemma] \label{lemm:main}

Let $\extil{X}$ be an interpretability logic and let 
%closed under taking subformulas. 
$\mathcal{C}$ be a (first or higher order) frame condition such that
for any \il-frame $F$ we have 
\[
F \models \mathcal{C} \Rightarrow F \models {\sf X}.
\] 
%Suppose that for any finite set of sentences $\mathcal{D}$ 
Let $\adset{D}$ be a finite set of sentences.
%there is a set 
Let $\mathcal{I}$ be a set of so-called \emph{invariants} 
of labeled frames so that we have 
the following properties.
\begin{itemize}
\item
$F \models \mathcal{I}^{\mathcal{U}} \Rightarrow
F \models \mathcal{C}$,
where $\mathcal{I}^{\mathcal{U}}$ is that part of $\mathcal{I}$ that is closed
under bounded unions of labeled frames.

\item
$\mathcal{I}$ contains the following invariant: 
$xRy \rightarrow \exists \, A {\in} (\nu (y) \setminus \nu (x))\cap 
\{ \Box \neg D \mid D$  a subformula of some  $B \in \adset{D} \}$.
%
%\item
%Any invariant in $\mathcal{I}$ holds on
%any one-point labeled frame.
%
\item
For any  adequate labeled frame  $F$, satisfying all the invariants,
we have the following.

\begin{itemize}

\item
Any \adset{D}-problem of $F$ can be eliminated by extending 
%the labeled frame 
$F$
in a way that conserves all invariants.

\item
Any \adset{D}-deficiency of $F$
%$\langle x,y, C\rhd D \rangle$ of $F$ such that $x\prec_B y$ 
can be eliminated by extending 
$F$
%the labeled frame 
in a way that conserves all invariants.
% and in a way that
%nd that moreover 
%no new deficiencies w.r.t. $x$ occur and 
%all additions lie $B$-critical above $x$.
\end{itemize}
\end{itemize}

In case such a set of invariants $\mathcal{I}$ exists, we have that
any \extil{X}-labeled adequate frame $F$ satisfying all the invariants 
%on which a 
%truth-lemma with respect to $\mathcal{D}$  is not violated 
can be extended to 
some labeled adequate \extil{X}-frame $\hat F$ on which a truth-lemma with 
respect to 
$\adset{D}$ holds.

Moreover, if for any finite \adset{D} that is closed under
subformulas and single negations, a corresponding set of  
invariants $\mathcal{I}$ can be found as above and such that moreover 
$\mathcal{I}$
holds on any one-point 
labeled frame, we have that $\extil{X}$ is a complete logic.
\end{lemma}

\begin{proof}
By subsequently eliminating problems and 
deficiencies by means of extensions. These elimination processes have to be robust in
the sense that 
%everything 
every problem or deficiency
that has been dealt with, should not possibly re-emerge. But, the requirements of the lemma 
almost immediately imply this.

For the second part of the Main Lemma,
we suppose that 
for any finite set \adset{D} closed under subformulas and single negations,
we can find a corresponding set of invariants $\mathcal{I}$.
If now, for any such \adset{D},
all the corresponding invariants $\mathcal{I}$ hold on any one-point
labeled frame, we are to see that \extil{X} is a complete logic, that is,
$\extil{X} \nvdash A \Rightarrow \exists M\; 
(M \models X \ \& \ M \models \neg A)$.

But this just follows from the above. 
If $\extil{X} \nvdash A$, we can find a maximal \extil{X}-consistent set 
$\Gamma$ with $\neg A \in \Gamma$. Let \adset{D} be the smallest set
that contains $\neg A$ and is closed under subformulas and single negations
and consider the invariants corresponding to \adset{D}.
The labeled frame $F\eqbydef \langle \{ x \} , \varnothing , \varnothing , 
\langle x, \Gamma \rangle \rangle$ can thus be extended to a labeled
adequate \extil{X}-frame $\hat F$ on which a truth lemma with respect to 
\adset{D} holds.
Thus certainly $\overline{\hat F} , x \Vdash \neg A$, that is, $A$ is not 
valid on the model induced by ${\hat F}$.
\end{proof}

The construction method can also be used to obtain decidability via the 
finite model property. In such a case, one should re-use worlds that were 
introduced earlier in the construction.

The following two lemmata indicate how good labels can be found for the
elimination of problems and deficiencies.

\begin{lemma}\label{lemm:problems}
Let $\Gamma$ be a maximal \extil{X}-consistent set such that $\neg (A\rhd B) \in \Gamma$.
Then there exists a maximal \extil{X}-consistent set $\Delta$ such that 
$\Gamma \crit{B} \Delta \ni A, \Box \neg A$.
\end{lemma}

\begin{proof}
So, consider $\neg (A\rhd B) \in \Gamma$, and suppose that no required $\Delta$ exists.
We can then find a\footnote{Writing out the definition and by compactness, we 
get a finite number of formulas 
$C_1, \ldots , C_n$ with $C_i \rhd B \in \Gamma$, 
such that $\neg C_1 ,\ldots ,\neg C_n, 
\Box \neg C_1, \ldots ,\Box \neg C_n , 
A , \Box \neg A \vdash_{\extil{X}} \bot$. We can now take $C \eqbydef C_1 \vee \ldots \vee 
C_n$. Clearly, as all the $C_i \rhd B \in \Gamma$, also $C\rhd B \in \Gamma$.}
formula $C$ for which $C\rhd B \in \Gamma$ such that
\[
\neg C , \Box \neg C , A , \Box \neg A \vdash_{\extil{X}} \bot .
\]
Consequently
\[
\vdash_{\extil{X}} A \wedge \Box \neg  A \rightarrow C \vee \Diamond C
\]
and thus, by Lemma \ref{lemm:basicil},
also $\vdash_{\extil{X}} A \rhd C$. But as $C\rhd B \in \Gamma$, also 
$A\rhd B \in \Gamma$. This clearly contradicts the consistency of $\Gamma$. 
\end{proof}

For deficiencies there is a similar lemma.

\begin{lemma}\label{lemm:deficiencies}
Consider $C \rhd D \in \Gamma \crit{B} \Delta \ni C$. There exists $\Delta'$
with $\Gamma \crit{B} \Delta'\ni D, \Box \neg D$. 
\end{lemma}

\begin{proof}
Suppose for a contradiction that $C \rhd D \in \Gamma \crit{B} \Delta \ni C$ and there 
does not exist a $\Delta'$
with $\Gamma \crit{B} \Delta'\ni D, \Box \neg D$. Taking the contraposition of 
Lemma \ref{lemm:problems} we get that $\neg (D \rhd B) \notin \Gamma$, whence
$D\rhd B \in \Gamma$ and also $C\rhd B \in \Gamma$. This clearly contradicts
the consistency of $\Delta$ as $\Gamma \crit{B} \Delta \ni C$. 
\end{proof}

\subsection{Completeness and the main lemma}

The main lemma provides a powerful method for proving modal completeness. In 
several cases it is actually the only known method available.

\begin{remark}\label{rema:application}
A modal completeness proof for an interpretability  logic \extil{X} is 
by the main lemma reduced to the following four 
ingredients.

\begin{itemize}
\item {\bf Frame Condition}
Providing a frame condition $\mathcal{C}$ and a proof that \[
F\models \mathcal{C} \Rightarrow F \models \extil{X}.
\]

%Providing a frame condition and a proof that this a frame condition of $\extil{X}$.
%Actually what we only need is one direction of the frame condition. Thus, we need
%$F\models \mathcal{C} \Rightarrow F\models \extil{X}$. Of course it is nicer
%if we can find $\mathcal{C}$ such that 
%$F\models \mathcal{C} \Leftrightarrow F\models \extil{X}$.

\item {\bf Invariants} Given a finite set of sentences \adset{D} (closed under 
subformulas and single negations),
providing invariants $\mathcal{I}$ that hold for any one-point labeled frame.
Certainly $\mathcal{I}$ should contain
$xRy \rightarrow \exists \, A {\in} (\nu (y) \setminus \nu (x))\cap 
\{ \Box D \mid D \in \adset{D} \}$.

\item {\bf elimination}
\begin{itemize}
\item {\bf Problems} Providing a procedure of elimination by extension for problems in 
labeled frames that satisfy all the invariants. This procedure should come with a
proof that it preserves all the  invariants.

\item {\bf Deficiencies} Providing a procedure of elimination by extension for deficiencies in 
labeled frames that satisfy all the invariants. Also this procedure should come with a
proof that it preserves all the  invariants.
\end{itemize}

\item {\bf Rounding up}
A proof that for any bounded chain of labeled frames that satisfy the invariants, automatically,
the union satisfies the frame condition $\mathcal{C}$ of the logic. 
\end{itemize}

\end{remark}

The completeness proofs that we will present will all have the same
structure, also in their preparations. As we will see, eliminating problems
is more elementary than eliminating deficiencies.

As we already pointed out,
we eliminate a problem by adding some new world plus an adequate label 
to the model we had. Like this, we get a structure that need not even be
an \il-model. For example, in general, the $R$ relation is not transitive.
To come back to at least an \il-model, we should close off the new structure
under transitivity of $R$ and $S$ et cetera. 
This closing off is in its self 
an easy and elementary process but we do want that the invariants are 
preserved under this process. Therefore we should have started already 
with a structure that admitted a closure.
%, fit to our purposes. 
Actually in this paper we will always want to obtain a model that
satisfies the frame condition of the logic.

The preparations to a completeness proof in this paper thus have the following 
structure.

%\begin{absorb} 
%Hier moeten emnul en wester nog aan worden toegevoegd: Evan!
%\end{absorb}

\begin{itemize}
\item
Determining a frame condition for \extil{X} and a corresponding notion of
an  \extil{X}-frame. 
%(Cf. Definition \ref{defi:frames} and
%Lemma \ref{lemm:ilsound} in the case of \il,  
%Definition \ref{defi:ilmframe} and Lemma \ref{lemm:frameilm} in the case
%of \ilm, Theorem \ref{theo:ilmoframe} in the case of \ilmo and Theorem \ref{homo}
%in the case of \extil{\sf{W}^*}.)

\item
Defining a notion of a quasi \extil{X}-frame. 
%(Cf. Definition \ref{defi:quasiframes}
%in the case of \il, Definition \ref{defi:ilmquasiframe} in the case of \ilm,
%Definition \ref{defi:ilmoquasi} in the case of \ilmo and
%Definition \ref{defi:ilwstarquasi} in the case of \extil{\sf{W}^*}.)

\item
Defining some notions that remain constant throughout the 
closing of quasi \extil{X}-frames, but somehow capture the
dynamic features of this process. 
%(Cf. critical and generalized cones; Definitions
%\ref{defi:critcone} and \ref{defi:gencone}, in the case of \il; 
%critical $\mathcal{M}$-cone, Definition
%\ref{defi:critmcone}, and $R\circ S$ in the case of \ilm and
%Definition \ref{defi:k} and Definition \ref{defi:ncone} both in the case of \ilmo and
%\extil{\sf{W}^*}.)

\item
Proving that a quasi \extil{X}-frame can be closed off to an
adequate labeled \extil{X}-frame. 
%(Cf. Lemma \ref{lemm:extension} and corollary
%\ref{coro:invariantil} for \il, Lemma \ref{lemm:ilmextension} and 
%Corollary \ref{coro:invariantilm} for \ilm, Lemma \ref{lemm:ilmoclosure} in the case
%of \ilmo and Lemma \ref{lemm:ilwstarclosure} in the case of \extil{\sf{W}^*}.)

\item
Preparing the elimination of deficiencies. 
%(Cf. Lemma \ref{lemm:deficiencies} in the case of \il,
%\ref{lemm:ilmdeficiencies} in the case of \ilm,
%Lemma \ref{lemm:ilmoexistence} in the case of \ilmo and
%Lemma \ref{lemm:ilwstarexistence} in the case of \extil{\sf{W}^*}.)
\end{itemize}

The most difficult job in a the completeness proofs we present in this paper, was
in finding correct invariants and in preparing the elimination of deficiencies. Once
this is fixed, the rest follows in a rather mechanical way. Especially the 
closure of quasi \extil{X}-frames to adequate \extil{X}-frames is a very laborious
enterprise.

\section{The logic \il}\label{sect:il}
The modal logic \il has been proved to be modally complete in 
\cite{JoVe90}. We shall reprove the completeness here using the main lemma.

The completeness proof of \il can be seen as the mother of all our completeness proofs
in interpretability logics.
Not only does it reflect the general structure of applications of the Main Lemma clearly, 
also it so that we
can use large parts of the preparations to the completeness proof of \il in other proofs too.
Especially closability proofs are cumulative. Thus, we can use the 
lemma that any quasi-frame is closable to an adequate frame, in any other completeness proof.

\subsection{Preparations}

\begin{definition}\label{defi:quasiframes}
A \emph{quasi-frame} $G$ is a quadruple $\langle W,R,S,\nu \rangle$.
Here $W$ is a non-empty set of worlds, and $R$ a binary relation on $W$.
$S$ is a set of binary relations on $W$ indexed by elements of $W$.
The $\nu$ is a labeling as defined on labeled frames. Critical cones
and generalized cones are defined just in the same way as in the
case of labeled frames.
$G$ should posess the following properties.
\begin{enumerate}
\item
$R$ is conversely well-founded

\item
$yS_xz \rightarrow xRy \ \& \ xRz$

\item
$xRy \rightarrow \nu (x) \sucs \nu (y)$

\item
$A \neq B \rightarrow \geone{A}{x} \cap \geone{B}{x} = \varnothing$

\item
$y{\in} \crone{A}{x} \rightarrow \nu (x) \crit{A} \nu (y)$

\end{enumerate}
 
\end{definition}
Clearly, adequate labeled frames are special cases of quasi frames.
Quasi-frames inherit all the notations from labeled frames.
In particular we can thus speak of chains and the like.

\begin{lemma}[\il-closure]\label{lemm:extension}
Let $G=\langle W,R,S,\nu \rangle$ be a quasi-frame. There is an adequate 
\il-frame $F$ extending $G$. That is, $F=\langle W,R',S',\nu \rangle$
with $R\subseteq R'$ and $S\subseteq S'$.
\end{lemma}

%\begin{absorb}
%Invarianten mee laten liften?
%\end{absorb}

\begin{proof}
 We define an
\emph{imperfection} on a quasi-frame $F_n$ to be a tuple $\gamma$ having one of
the following forms.
\begin{itemize}
\item[$(\romannumeral 1)$]
$\gamma = \langle 0,a,b,c \rangle$ with 
$F_n \models aRbRc$ but $F_n \not \models aRc$

\item[$(\romannumeral 2)$]
$\gamma = \langle 1,a,b\rangle$ with $F_n\models aRb$ but  $F_n \not \models bS_ab$

\item[$(\romannumeral 3)$]
$\gamma = \langle 2,a,b,c,d \rangle$ 
with $F_n \models bS_acS_ad$ but not $F_n \models bS_ad$

\item[$(\romannumeral 4)$]
$\gamma = \langle 3,a,b,c\rangle$
with $F_n\models aRbRc$ but $F_n \not \models bS_ac$

\end{itemize}
Now let us start with a quasi-frame $G=\langle W,R,S,\nu \rangle$.
We will define a chain of quasi-frames. Every new element in the chain will
have at least one imperfection less than its predecessor. The union will have no
imperfections at all. It will be our required adequate \il-frame.
\medskip

Let $<_0$ be the well-ordering on
\[
C \eqbydef 
(\{ 0\} \times W^3) \cup
(\{ 1\} \times W^2) \cup
(\{ 2\} \times W^4) \cup
(\{ 3\} \times W^3)
\]
induced by the occurrence order in some fixed enumeration of $C$. 
%(Enumerations
%are always of type $\omega$.) 
We define our chain to start with

$F_0\eqbydef G$.
To go from $F_n$ to $F_{n+1}$ we proceed as follows. Let $\gamma$ be the 
$<_0$-minimal imperfection on $F_n$. In case no such $\gamma$ exists we set 
$F_{n+1}\eqbydef F_n$. If such a $\gamma$ does exist, $F_{n+1}$ is as
dicted by the case distinctions.
\begin{itemize}
\item[$(\romannumeral 1)$]
$F_{n+1}\eqbydef \langle W_n , R_n \cup \{  \langle a,c \rangle \}, S_n , \nu_n \rangle$
\item[$(\romannumeral 2)$]
$F_{n+1}\eqbydef \langle W_n , R_n , S_n \cup \{  \langle a,b,b \rangle \}, \nu_n \rangle$
\item[$(\romannumeral 3)$]
$F_{n+1}\eqbydef \langle W_n , R_n , S_n \cup \{  \langle a,b,d \rangle \} , \nu_n \rangle$
\item[$(\romannumeral 4)$]
$F_{n+1}\eqbydef \langle W_n , R_n \cup \{  \langle a,c \rangle \}, 
S_n \cup \{  \langle a,b,c \rangle \} , \nu_n \rangle$
\end{itemize} 
By an easy but elaborate induction, we can see that each $F_n$ is a
quasi-frame. The induction boils down to checking for each case
$(\romannumeral 1)$-$(\romannumeral 4)$ that all the properties
$(1)$-$(5)$ from Definition \ref{defi:quasiframes} remain valid.

Instead of proving $(4)$ and $(5)$, it is better to prove something stronger, 
that is, that the critical and generalized cones remain unchanged.

\begin{itemize}
\item[4'.]
$\forall n \ 
[F_{n+1} \models y {\in} \geone{A}{x} 
\Leftrightarrow F_n \models y {\in} \geone{A}{x}]$

\item[5'.]
$\forall n \ 
[F_{n+1} \models y {\in} \crone{A}{x} 
\Leftrightarrow F_n \models y {\in} \crone{A}{x}]$
\end{itemize}
Next, it is 
not hard to prove that $F\eqbydef \cup_{i\in \omega}F_i$ is the 
required adequate \il-frame. To 
this extent, the following 
properties have to be checked. All properties have easy proofs.
\[
\begin{array}{ll}
(a.)\  W \mbox{ is the domain of  }F   & 
(g.)\  F\models xRy \rightarrow yS_x y\\
(b.)\  R_0 \subseteq \cup_{i\in \omega} R_i   & 
(h.)\  F\models xRyRz \rightarrow yS_x z\\
 (c.)\  S_0 \subseteq \cup_{i\in \omega} S_i  & 
(i.) \ F\models uS_x v S_x w \rightarrow u S_x w\\
 (d.) \ R \mbox{ is conv. wellfounded on }F  & 
(j.)  \ F\models xRy \Rightarrow \nu (x) \sucs \nu (y)\\
 (e.)\ F\models xRyRz \rightarrow xR z   & 
(k.) \ A \neq B \Rightarrow F \models \geone{A}{x} \cap \geone{B}{x} = \varnothing\\
 (f.)\ F\models yS_xz \rightarrow xRy \ \& \ xRz   & 
(l.) \ F\models y{\in} \crone{A}{x} \Rightarrow \nu (x) \crit{A} \nu (y)\\
\end{array}
\]
\end{proof}

We note that the \il-frame $F\supseteq G$ from above is actually the 
minimal one extending $G$. If in the sequel, if we refer to the closure
given by the lemma, we shall mean this minimal one. Also do we note 
that the proof is independent on the enumeration of 
$C$ and hence the order $<_0$ on $C$. 
The lemma can also be applied to non-labeled structures. If we drop all the
requirements on the labels in Definition \ref{defi:quasiframes} and in 
Lemma \ref{lemm:extension} we end up with a true statement 
about just the old \il-frames.

Lemma \ref{lemm:extension} also 
allows a very short proof running as follows. Any intersection 
of adequate \il-frames with the same domain is again an adequate \il-frame.
There is an adequate \il-frame extending $G$. Thus by taking intersections
we find a minimal one. We have chosen to present our explicit proof 
as they allow us, now and in the sequel, to see which properties remain invariant.
\medskip

\begin{corollary}\label{coro:invariantil}
Let \adset{D} be a finite set of sentences, closed under subformulas and single negations.
Let $G=\langle W,R,S,\nu \rangle$ be a quasi-frame on which
\[
xRy \rightarrow \exists \,  A {\in}((\nu (y)\setminus \nu{x})\cap 
\{  \Box D \mid D \in \adset{D}\} ) \ \ \ (*)
\]
holds. Property $(*)$ does also hold on the \il-closure $F$ of $G$.
\end{corollary}

\begin{proof}
We can just take the property along in the proof of Lemma \ref{lemm:extension}.
In Case $(\romannumeral 1)$ and $(\romannumeral 4)$ we note that
$aRbRc \rightarrow \nu (a) \boxin \nu (c)$. Thus, if
$A {\in}((\nu (c)\setminus \nu (b) )\cap 
\{  \Box D \mid D \in \adset{D}\} )$, then certainly $A \not \in \nu (a)$.
\end{proof}

%All completeness proofs we will expose will have the same general structure.
%As an example we will now proof the modal completeness of \il.
We have now done all the preparations for the completeness proof. Normally,
also a lemma is needed to deal with deficiencies. But in the case of \il,
Lemma \ref{lemm:deficiencies} suffices.

\subsection{Modal completeness}

\begin{theorem}
\il is a complete logic
\end{theorem}

\begin{proof}
We specify the four ingredients mentioned in Remark \ref{rema:application}.
%\begin{itemize}
%\item
\medskip

{\bf Frame Condition} For \il, the frame condition is empty, that is, every 
frame is an \il frame.
\medskip

{\bf Invariants}
Given a finite set of sentences \adset{D} closed under subformulas and single 
negation, the only invariant is 
$xRy \rightarrow \exists \, A {\in} (\nu (y) \setminus \nu (x))\cap 
\{ \Box D \mid D \in \adset{D} \}$. Clearly this invariant holds on any
one-point labeled frame.
\medskip

{\bf Elimination} 
So, let $F \eqbydef \langle W , R, S, \nu \rangle$ be a labeled frame satisfying the 
invariant. We will see how to eliminate both problems and deficiencies while 
conserving the invariant.

\medskip
{\bf Problems} Any problem $\langle a, \neg (A \rhd B)\rangle$ of $F$ will
be eliminated in two steps.

\begin{enumerate}
\item
With Lemma \ref{lemm:problems} we find $\Delta$ with 
$\nu (a) \crit{B} \Delta \ni A, \Box \neg A$. We fix some $b\notin W$. We now define
\[
G':= \langle W\cup \{  b\} , R \cup \{ \langle a, b\rangle \}, S, \nu \cup
\{ \langle  b, \Delta \rangle , \langle \langle a,b \rangle, B \rangle \}\rangle.
\]
It is easy to see that $G'$ is actually a quasi-frame. Note that if 
$G' \models xRb$, then $x$ must be $a$ and whence $\nu (x) \sucs \nu (b)$
by definition of $\nu (b)$. Also it is not hard to see that 
if $b\in \crone{C}{x}$ for $x{\neq} a$, that then $\nu (x) \crit{C}\nu (b)$. 
For, $b\in \crone{C}{x}$ implies
$a \in \crone{C}{x}$ whence $\nu (x) \crit{C} \nu (a)$. By $\nu (a) \sucs \nu (b)$
we get that $\nu (x) \crit{C} \nu (b)$. In case $x{=}a$ we see that by definition
$b \in \crone{B}{a}$. But, 
%also by definition 
we have chosen $\Delta$ so that
$\nu (a) \crit{B} \nu (b)$. We also 
see that $G'$ satisfies the invariant as 
$\Box \neg A \in \nu (b) \setminus \nu (a)$ and $\sim A \in \adset{D}$.

\item
With Lemma \ref{lemm:extension} we extend $G'$ to an adequate labeled \il-frame
$G$. 
Corollary \ref{coro:invariantil} tells us that the invariant indeed holds at $G$.
Clearly $\langle a, \neg (A \rhd B ) \rangle$ is no longer a problem in 
$G$.

%We want to see that the invariant also holds in $G$. To this extend, we
%just take the invariant along in the proof of Lemma \ref{lemm:extension}. This
%is fairly routine.

\end{enumerate}
\medskip

{\bf Deficiencies}. Again, any deficiency $\langle a,b, C \rhd D\rangle$ in 
$F$ will be eliminated in two steps.
\begin{enumerate}
\item
We first define $B$ to be the formula such that $b\in \crone{B}{a}$. If such a
$B$ does not exist, we take $B$ to be $\bot$. Note that if such a $B$ does exist,
it must be unique by Property $4$ of Definition \ref{defi:quasiframes}.
By Lemma \ref{lemm:deficiencies} we can now find a 
$\Delta'$ such that
$\nu (a) \crit{B} \Delta' \ni D, \Box \neg D$. We fix some $c \not \in W$
and define
\[
G' \eqbydef \langle W, R \cup \{ a,c \}, 
S \cup \{ a,b,c \}, \nu \cup \{  c, \Delta' \} \rangle .
\]
Again it is not hard to see that $G'$ is a quasi-frame that
satisfies the invariant. Clearly $R$ is conversely well-founded.
The only new $S$ in $G'$ is $bS_ac$,
% Conformly 
but 
we 
also
defined $aRc$.
We have chosen 
$\Delta'$ such that $\nu(a)\crit{B}\nu(c)$. Clearly 
$\Box \neg D \not \in \nu (a)$.

\item
Again, $G'$ is closed off under the frame conditions with Lemma
\ref{lemm:extension}. Again we note that the invariant is preserved 
in this process. Clearly $\langle a,b, C \rhd D\rangle$ is not a deficiency 
in $G$.
\medskip

{\bf Rounding up}
Clearly the union of a bounded chain of \il-frames is again an \il-frame.

\end{enumerate}

\end{proof}
It is well known that \il has the finite model property and whence is decidable.
With some more effort however we could have obtained the finite model 
property using the Main Lemma. We have chosen not to do so, as 
for our purposes the completeness via the construction method is sufficient.

Also, to obtain the finite model property, one has to re-use worlds during the 
construction method. The constraints on which worlds can be re-used is per logic 
differently. One aim of this section was to 
prove some results on a construction that is present in all other completeness proofs too.
Therefore we needed some uniformity and did not want to consider re-using of worlds.

%
%\section{The Logic \ilm}
%\input{ilmcomplete}
%
%\section{Essentially $\Sigma_1$-sentences of \ilm}\label{sect:essigma}
%\input{sigmasentences}

\section{The logic \ilmo}

%
% Local preamble
%

\newcommand{\const}[1]{{\bf #1}}

\renewcommand{\theenumi}{\arabic{enumi}.}
\renewcommand{\labelenumi}{\theenumi}

%\theoremstyle{definition}
%\newtheorem{subclaiminlemma}{Claim}
%\renewcommand{\thesubclaiminlemma}{\thelemma\alph{subclaiminlemma}}

%\theoremstyle{theorem}
%\newtheorem*{corollarynn}{Corollary}

%
% Section ILM0
%
%
%To start, let us recall the schema \mo.
%
%\begin{itemize}
%\item[\mo] $A\rhd B\rightarrow\Diamond A\wedge\Box C\rhd B\wedge\Box C$
%\end{itemize}
%

This section is devoted to showing the following theorem.
%\footnote{A 
%proof sketch of this theorem was first given in \cite{joo98}.}

\begin{theorem}\label{theo:completenessilmo}
\ilmo is a complete logic.
\end{theorem}

It turns out that the modal frame condition of \ilmo gives rise to a bewildering structure of possible models that seems very hard to tame. As \principel{M_0} is in \ilal, it is important that the class of \ilmo-frames is well understood. For a long time \ilwstar has been conjectured (\cite{Vis91}) to be \ilal. A first step in proving this conjecture would have been a modal completeness proof of \ilwstar. 

It is well known that \ilwstar is the union of \ilw and \ilmo, see Lemma \ref{lemm:emnulsteriswester}. The modal completeness of \ilw was proved in \cite{jonvelt99}. So, the missing link was a modal completeness proof for \ilmo. In \cite{joo98} a proof sketch of this completeness result was given. In this paper we give for the first time a fully detailed proof. 

In the light of Remark \ref{rema:application} a proof of Theorem
\ref{theo:completenessilmo} boils down to giving the four
ingredients mentioned there. Sections
\ref{sect:ilmo:framecondition}, \ref{sect:ilmo:invariants},
\ref{sect:ilmo:problems}, \ref{sect:ilmo:deficiencies}
and \ref{sect:ilmo:roundingup} below
contain those ingredients.
Before these main sections, we have in Section \ref{sect:ilmo:preliminaries}
some preliminaries.
We start in Section
\ref{sect:ilmo:diffi} with an overview of the difficulties we encounter
during the application of the construction method to \ilmo.

\subsection{Overview of difficulties}\label{sect:ilmo:diffi}

%We shortly review the construction method and
%see what difficulties we might encounter when we apply it to
%\ilmo.%
%
%Roughly the construction proceeds as follows.
%We identify a problem or a deficiency in a labeled frame $F$.
%We solve this problem or deficiency by extending $F$ to a
%labeled frame $F'$. Then we extend $F'$ to a frame $F''$
%in which all the frame conditions for the logic under consideration
%hold. We repeat this procedure until no problems nor deficiencies
%are present any more.

In the construction method we repeatedly eliminate problems
and deficiencies by extensions that satisfy all the invariants.
During these operations we need to keep track of two things.

\begin{enumerate}
\item\label{item:dd0} If $x$ has been added to solve a problem
  in $w$, say $\neg(A\rhd B)\in\nu(w)$.
  Then for all $y$ such that $xS_wy$ we have $\nu(w)\crit{B}\nu(y)$.
\item\label{item:dd1} If $wRx$ then $\nu(w)\sucs\nu(x)$
\end{enumerate}

Item \ref{item:dd0} does not impose any direct difficulties.
But some do emerge when we try to deal with the difficulties
concerning Item \ref{item:dd1}
So let us see why it is difficult to ensure \ref{item:dd1}
Suppose we have $wRxRyS_wy'Rz$. The \mo--frame condition (Theorem 
\ref{theo:ilmoframe}) requires that
we also have $xRz$. So, from \ref{item:dd1} and the \mo--frame condition we
obtain
$wRxRyS_wy'Rz\rightarrow\nu(x)\sucs\nu(z)$.
%
%If we put $\Delta\boxin\Delta':\Leftrightarrow \{\Box A\mid \Box A\in\Delta\}\subseteq\Delta'$
%then a 
A sufficient (and in certain sense necessary) condition is,
\[wRxRyS_wy'\rightarrow\nu(x)\boxin\nu(y').\]
Let us illustrate some difficulties concerning this condition by some examples.
Consider the left model in Figure \ref{fig:ilmoex00}.
That is, we have a deficiency in $w$ w.r.t. $y$. Namely, $C\rhd D\in\nu(w)$ and $C\in\nu(y)$.
If we solve this deficiency by adding a world $y'$, we thus require that
for all $x$ such that $wRxRy$ we have $\nu(x)\boxin\nu(y')$.
This difficulty is partially handled by  Lemma \ref{lemm:vetweleenlabel} below.
We omit a proof, but it can easily be given by replacing in the corresponding
lemma for \ilm, applications of the ${\sf M}$-axiom by applications of the ${\sf M}_0$-axiom.

\begin{lemma}\label{lemm:vetweleenlabel}
Let $\Gamma,\Delta$ be \mcs's such that $C\rhd D\in \Gamma$,
$\Gamma\crit{A}\Delta$ and $\Diamond C\in\Delta$. Then there
exists some $\Delta'$ with $\Gamma\crit{A}\Delta'$, $\Box\neg
D,D\in\Delta'$ and $\Delta\boxin\Delta'$.
\end{lemma}

\begin{figure}
\begin{center}
\resizebox{!}{5cm}{\input{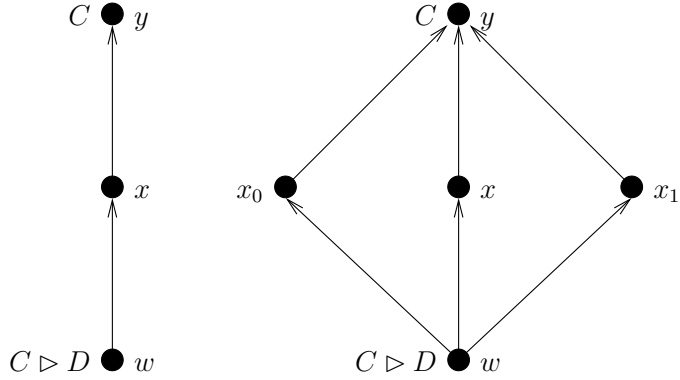}}
\end{center}
\caption{A deficiency in $w$ w.r.t. $y$}
\label{fig:ilmoex00}
\end{figure}

Let us now consider the right most model in 
Figure \ref{fig:ilmoex00}. 
We have at least for two different worlds $x$, say $x_0$ and $x_1$, that $wRxRy$.
Lemma \ref{lemm:vetweleenlabel} is applicable to $\nu(x_0)$ and $\nu(x_1)$ separately but not simultaneously.
In other words we find $y'_0$ and $y'_1$ such that $\nu(x_0)\boxin\nu(y'_0)$ and $\nu(x_1)\boxin\nu(y'_1)$.
But we actually want one single $y'$ such that $\nu(x_0)\boxin\nu(y')$ and $\nu(x_1)\boxin\nu(y')$.
We shall handle this difficulty by ensuring 
that it is enough to consider only one of the worlds
in between $w$ and $y$. To be precise, we shall ensure 
$\nu(x')\boxin\nu(x)$ or $\nu(x)\boxin\nu(x')$.

%But now some difficulties concerning Item \ref{item:dd0} occur.
%In the situations in Figure \ref{fig:ilmoex00} we were asked to solve
%a deficiency in $w$ w.r.t. $y$.
%But we actually solved one in $w$ w.r.t. some $x$ in between $w$ and $y$.
%That is to say, we applied Lemma \ref{lemm:vetweleenlabel} to $\Diamond C \in x$ to 
%obtain the required $y'$. Critical cones should be respected. Thus, if 
%$y\in \crone{A}{w}$ we can only use Lemma \ref{lemm:vetweleenlabel}
%only if we know that also $x\in \crone{A}{w}$. 

But now some difficulties concerning Item \ref{item:dd0} occur.
In the situations in Figure \ref{fig:ilmoex00} we were asked to solve
a deficiency in $w$ w.r.t. $y$.
As usual, if $w\crit{A}y$ then we should be ably to choose a solution $y'$ such that
$w\crit{A}y'$. But Lemma \ref{lemm:vetweleenlabel} takes only criticallity of $x$ w.r.t.\
$w$ into account.
This issue is solved by ensuring that $wRxRy\in\crone{A}{w}$ implies $\nu(w)\crit{A}\nu(x)$.

%We can let $y'$ solve both deficiencies if we have that $\nu(w)\crit{A}\nu(y')$ whenever $y\in\crone{A}{w}$.
%And this is assured 

\begin{figure}
\begin{center}
\resizebox{!}{5cm}{\input{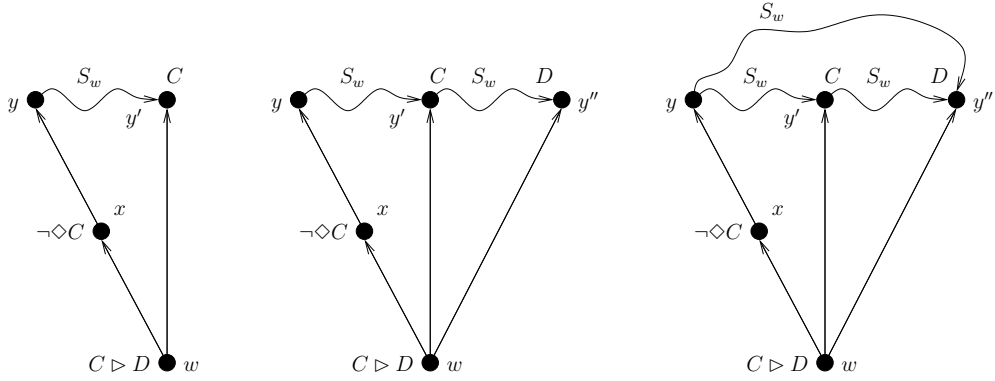}}
\end{center}
\caption{A deficiency in $w$ w.r.t.\ $y'$}
\label{fig:ilmoex01}
\end{figure}

We are not there yet.
Consider the leftmost model in Figure \ref{fig:ilmoex01}.
That is, we have a deficiency in $w$ w.r.t. $y'$.
Namely, $C\rhd D\in\nu(w)$ and $C\in\nu(y')$.
If we add a world $y''$ to solve this deficiency, as in the middle model, then
by transitivity of $S_w$ we have $yS_wy''$, as shown in the rightmost model.
So, we require that $\nu(x)\boxin\nu(y'')$.
But we might very well have $\Diamond C\not\in\nu(x)$.
So the Lemma \ref{lemm:vetweleenlabel} is not applicable.

In Lemma \ref{lemm:ilmoexistence} we formulate and prove a more complicated version of the Lemma 
\ref{lemm:vetweleenlabel}
which basically says that if we have chosen $\nu(y')$ appropriately,
then we can choose $\nu(y'')$ such that $\nu(x)\boxin\nu(y'')$.
And moreover, Lemma \ref{lemm:ilmoexistence} 
ensures us that we can, indeed, choose $\nu(y')$ appropriate.

\subsection{Preliminaries}\label{sect:ilmo:preliminaries}

\begin{definition}[$\trans{T}$, $T^*$, $T;T'$, $T^1$, $T^{\geq 2}$, $T\cup T'$]
Let $T$ and $T'$ be binary relations on a set $W$. We fix the following fairly standard
notations.
$\trans{T}$ is the transitive closure of $T$; $T^*$ is the transitive 
reflexive closure of $T$;
$xT\comp T'y\Leftrightarrow \exists t\;xTtT'y$;
$xT^1y\Leftrightarrow xTy\wedge\neg\exists t\;xTtTy$;
$xT^{\geq 2}y\Leftrightarrow xTy\wedge\neg(xT^1y)$ and 
$x T\cup T' y\Leftrightarrow xTy\vee xT'y$.
\end{definition}

\begin{definition}[$\spure{w}{S}$]
Let $F=\langle W,R,S,\nu\rangle$ be a quasi--frame. For each $w\in
W$ we define the relation $\spure{w}{S}$, of pure $S_w$
transitions, as follows.
\[
x\spure{w}{S}y\Leftrightarrow xS_wy\wedge\neg(x=y)\wedge\neg
(x(S_w\cup R)^*\comp R \comp (S_w\cup R)^*y)
\]
\end{definition}

\begin{definition}[Adequate \ilmo--frame]\label{defi:ilmoframecondition}
Let $F=\langle W,R,S,\nu\rangle$ be an adequate frame. We say that
$F$ is an adequate \ilmo--frame iff.\ the following additional
properties hold.\footnote{One might think that
\ref{i3:ilmoframecondition} is superfluous. In finite frame this
is indeed the case, but in the general case we need it as an
requirement.}
\begin{enumerate}
\addtocounter{enumi}{3}
\item\label{i1:ilmoframecondition} $wRxRyS_wy'Rz\rightarrow xRz$
\item\label{i2:ilmoframecondition} $wRxRyS_wy'\rightarrow\nu(x)\boxin\nu(y')$
\item\label{i3:ilmoframecondition} $xS_wy\rightarrow x(\spure{w}{S}\cup R)^*y$
\item\label{i4:ilmoframecondition} $xRy\rightarrow x\trans{(R^1)}y$
%\item\label{i5:ilmoframecondition} $wKxK^1y\trans{\spure{w}{S}}y'K^1z\rightarrow xK^1z$
\end{enumerate}
\end{definition}

As usual, when we speak of \ilmo--frames we shall actually mean an
adequate \ilmo--frame. Below we will construct \ilmo--frames out
of frames belonging to a certain subclass of the class of
quasi--frames. (Namely the quasi--\ilmo--frames, see Definition
\ref{defi:ilmoquasi} below.) We would like to predict on forehand
which extra $R$ relations will be added during this construction.
The following definition does just that.

\begin{definition}[$K(F)$, $K$]\label{defi:k}
Let $F=\langle W,R,S,\nu\rangle$ be a quasi--frame. We define
$K=K(F)$ to be the smallest binary relation on $W$ such that
\begin{enumerate}
\item\label{item0:defi:k} $R\subseteq K$,
\item\label{item1:defi:k} $K=K^{\textrm{tr}}$,
\item\label{item2:defi:k} $wKxK^1y\trans{(\spure{w}{S})}y'K^1z\rightarrow xKz$.
\end{enumerate}
\end{definition}

Note that for \ilmo--frames we have $K=R$.
The following lemma shows that $K$ satisfies some stability
conditions. The lemma will mainly be used to show that whenever we
extend $R$ within $K$, then $K$ does not change.

\begin{lemma}\label{lemm:kstabiel}
Let $F_0=\langle W,R_0,S,\nu\rangle$ and $F_1=\langle
W,R_1,S,\nu\rangle$ be quasi--frames.
If $R_1\subseteq K(F_0)$ and $R_0\subseteq K(F_1)$. Then
$K(F_0)=K(F_1)$.
\end{lemma}

In a great deal of situations we have a particular interest in
$K^1$. To determine some of its properties the following lemma
comes in handy. It basically shows that we can compute $K$
by first closing of under the \mo--condition and then
take the transitive closure.

\begin{lemma}[Calculation  of $K$]\label{lemm:k1calc}
Let $F=\langle W,R,S,\nu\rangle$ be a quasi--frame.
Let $K=K(F)$ and suppose $K$ conversely well--founded.
Let $T$ be a binary relation on $W$ such that

\begin{enumerate}
\item\label{i0:lemm:k1calc} $R\subseteq \trans{T}\subseteq K$,
\item\label{i1:lemm:k1calc} $w\trans{T}xT^1y\trans{(\spure{w}{S})}y'T^1z\rightarrow x\trans{T}z$.
\end{enumerate}

Then we have the following.
{
\renewcommand{\theenumi}{(\alph{enumi})}
\renewcommand{\labelenumi}{\theenumi}
\begin{enumerate}
\item\label{r0:lemm:k1calc} $K=\trans{T}$
\item\label{r1:lemm:k1calc} $xK^1y\rightarrow xTy$
\end{enumerate}
}
\end{lemma}

\begin{proof}
To see \ref{r0:lemm:k1calc}, it is enough to see that $\trans{T}$
satisfies the three properties of the definition of $K$
(Definition \ref{defi:k}). Item \ref{r1:lemm:k1calc} follows from
\ref{r0:lemm:k1calc}. 
\end{proof}

Another entity that changes during the construction of an \ilmo--frame
out of a quasi--frame is the critical cone 
%(which is defined for all logics \ilx,
%see Definition \ref{defi:critcone}). 
In accordance with the above
definition of $K(F)$, we also like to predict what eventually
becomes the critical cone.

\newcommand{\ncone}[2]{\ensuremath{\mathcal{N}^{#1}_{#2}}\xspace}
\begin{definition}[\ncone{C}{w}]\label{defi:ncone}
For any quasi--frame $F$ we define \ncone{C}{w} to be the smallest set such that
\begin{enumerate}
\item $\nu(w,x)=C\Rightarrow x\in\ncone{C}{w}$,
\item $x\in\ncone{C}{w}\wedge x(K\cup S_w)y\Rightarrow y\in\ncone{C}{w}$.
\end{enumerate}
\end{definition}

In accordance with the notion of a quasi--frame we introduce the notion
of a quasi--\ilmo--frame. This gives sufficient conditions for
a quasi--frame to be closeable, not only under the \il--frameconditions,
but under all the \ilmo--frameconditions.

\begin{definition}[Quasi--\ilmo--frame]\label{defi:ilmoquasi}
A quasi--\ilmo--frame is a quasi--frame that satisfies the following additional
properties.

\begin{enumerate}
\addtocounter{enumi}{5}
\item\label{ilmoquasi:1}$K$ is conversely well--founded.
\item\label{ilmoquasi:2}$xKy\rightarrow\nu(x)\sucs\nu(y)$
\item\label{ilmoquasi:4}$x\in\ncone{A}{w}\rightarrow \nu(w)\crit{A}\nu(x)$
\item\label{ilmoquasi:6}$wKxKy(S_w\cup K)^*y'\rightarrow\nu(x)\boxin\nu(y')$
\item\label{ilmoquasi:7}$xS_wy\rightarrow x(\spure{w}{S}\cup R)^*y$
\item\label{ilmoquasi:9}$wKxK^1y\trans{(\spure{w}{S})}y'K^1z\rightarrow x\trans{(K^1)}z$
\item\label{ilmoquasi:8}$xRy\rightarrow x\trans{(R^1)}y$
\end{enumerate}
\end{definition}

\begin{lemma}\label{lemm:r1k1}
If $F$ is a quasi--\ilmo--frame, then $K=\trans{(K^1)}$.
\end{lemma}
\begin{proof}
Using Lemma \ref{lemm:k1calc}. 
\end{proof}

\begin{lemma}\label{lemm:klemma1}
Suppose that $F$ is a quasi--\ilmo--frame. Let $K=K(F)$. Let $K'$,
$K''$ and $K'''$ the smallest binary relations on $W$ satifying
\ref{item0:defi:k} and \ref{item1:defi:k} of \ref{defi:k} and
additionaly we have the following.
{

\begin{enumerate}
\renewcommand{\theenumi}{\arabic{enumi}$'$.}
\renewcommand{\labelenumi}{\theenumi}
\addtocounter{enumi}{2}
\item\label{item2:defi:k_} $wK'x{K'}^1y(\spure{w}{S}\cup K')^*y'{K'}^1z\rightarrow xK'z$
\renewcommand{\theenumi}{\arabic{enumi}$''$.}
\renewcommand{\labelenumi}{\theenumi}
\addtocounter{enumi}{-1}
\item\label{item2:defi:k__} $wK''xK''y\trans{(\spure{w}{S})}y'K''z\rightarrow xK''z$
\renewcommand{\theenumi}{\arabic{enumi}$'''$.}
\renewcommand{\labelenumi}{\theenumi}
\addtocounter{enumi}{-1}
\item\label{item2:defi:k___} $wK'''xK'''y(S_w\cup K''')^*y'K'''z\rightarrow xK'''z$
\end{enumerate}
} Then $K=K'=K''=K'''$.
\end{lemma}

\begin{proof}
Using
Lemma \ref{lemm:r1k1}. 
\end{proof}

Before we move on, let us first sum up a few comments.

\begin{corollary}
If $F=\langle W,R,S,\nu\rangle$ is an adequate \ilmo--frame. Then
we have the following.
\begin{enumerate}
\item $K(F)=R$
\item $F\models x\in\ncone{A}{w}\Leftrightarrow F\models
x\in\crone{A}{w}$
\item $F$ is a quasi--\ilmo--frame
\end{enumerate}
\end{corollary}

\begin{lemma}[\ilmo--closure]\label{lemm:ilmoclosure}
Any quasi--\ilmo--frame can be extended to an adequate
\ilmo--frame.
\end{lemma}

\begin{proof}

Given a quasi--\ilmo--frame $F$ we construct a sequence
\[F=F_0\subseteq F_1\subseteq\cdots\]
very similar to the sequence
constructed for the \il closure of a quasi--frame (Lemma \ref{lemm:extension}). The
only difference is that we add a fifth entry to the list of imperfections.
{\renewcommand{\theenumi}{(\roman{enumi})}
\renewcommand{\labelenumi}{\theenumi}
\begin{enumerate}
\addtocounter{enumi}{4}
\item\label{imp:mo}
$\gamma=\langle 4,w,a,b,b',c\rangle$ with $F_n\models wRaRbS_wb'Rc$ but $F_n\not\models aRc$
\end{enumerate}}
\noindent
In this case we set, of course, 
$F_{n+1} \eqbydef \pair{W_n, R_n \cup \pair{a,c}, S_n, \nu_n}$.
First we will show by induction that each $F_n$ is a quasi--\ilmo--frame. Then
we show that the union $\hat F=\bigcup_{n\geq 0}F_n$, is quasi and
satisfies all the \ilmo frame conditions.

We assume that $F_n$ is a quasi-\ilmo-frame and define
$K^n \eqbydef K(F_n)$,
$R^n\eqbydef R^{F_n}$ and 
$S^n \eqbydef S^{F_n}$. 
Quasi-ness of $F_{n+1}$ will follow from Claim \ref{subclaim0:moclos}, and
from Claim \ref{subclaim2:moclos} we may conlude that $F_{n+1}$ is indeed a quasi-\ilmo-frame.
\begin{claim}\label{subclaim0:moclos}
For all $w,x,y$ and $A$ we have the following.
{
\renewcommand{\theenumi}{(\alph{enumi})}
\renewcommand{\labelenumi}{\theenumi}
\begin{enumerate}
\item\label{i1:subclaim0:moclos} $R^{n+1}\subseteq K^n$
\item\label{i2:subclaim0:moclos} $x(S_w^{n+1}\cup R^{n+1})^*y\Rightarrow x(S_w^{n}\cup K^n)^* y$
\item\label{i3:subclaim0:moclos} $F_{n+1}\models x\in\crone{A}{w}\Rightarrow F_n\models x\in\ncone{A}{w}$.
\end{enumerate}
}
\end{claim}

\begin{proof}
We distinguish cases according to which imperfection is dealt with in the step
from $F_n$ to $F_{n+1}$. The only interesting case is the new imperfection which is dealt with
by Lemma \ref{lemm:klemma1}, Item \ref{item2:defi:k__}
\end{proof}

\begin{claim}\label{subclaim2:moclos}
For all $w,x$ and $A$ we have the following.
\begin{enumerate}
\item\label{i1:subclaim2:moclos} $K^{n+1}\subseteq K^n$.
\item\label{i2:subclaim2:moclos} $x(S^{n+1}_w\cup K^{n+1})^*y\Rightarrow x(S^n_w\cup K^n)^* y$
\item\label{i3:subclaim2:moclos} $F_{n+1}\models x\in\ncone{A}{w}\Rightarrow F_n\models x\in\ncone{A}{w}$.
\end{enumerate}
\end{claim}

\begin{proof}
Item \ref{i1:subclaim2:moclos} follows by Claim \ref{subclaim0:moclos} and Lemma \ref{lemm:kstabiel}.
Item \ref{i2:subclaim2:moclos} follows from Item \ref{i1:subclaim2:moclos} and
Claim \ref{subclaim0:moclos}-\ref{i2:subclaim0:moclos}.
Item \ref{i3:subclaim2:moclos} is an immediate corollary of item \ref{i2:subclaim2:moclos}
\end{proof}
Again, it is not hard to see that $\hat F=\bigcup_{n\geq 0}F_n$ is an adequate \ilmo-frame.
\end{proof}

\newcommand{\kstepsfull}{\{(x,y)\mid\exists z\,(\nu(x)\boxin\nu(z)\wedge x(R\cup S)^*zRy)\}}
\begin{lemma}\label{lemm:kstep}
Let $F=\langle W,R,S,\nu\rangle$ be a quasi--\ilmo--frame and $K=K(F)$.
Then
\[xKy\rightarrow \exists z\,(\nu(x)\boxin\nu(z)\wedge x(R\cup S)^*zRy).\]
\end{lemma}

\begin{proof}

We define
\newcommand{\ksteps}{T}
$ T \eqbydef \kstepsfull$.
It is not hard to see that $T$ is transitive and that 
$\{(x,y)\mid\exists t\,(\nu(x)\boxin\nu(t)\wedge xT\comp(S\cup K)^*tTy)\}\subseteq T$.
We now define $K'=K\cap T$. 
We have to show that $K'=K$. As $K'\subseteq K$ is trivial, we will show $K\subseteq K'$.

It is easy to see that $K'$ satisfies properties \ref{item0:defi:k}, \ref{item1:defi:k} 
and \ref{item2:defi:k}
of Definition \ref{defi:k}; It follows on the two observations on $T$ we just made.
Since $K$ is the smallest binary relation that satisfies these properties we conclude
$K\subseteq K'$.
\end{proof}

The next lemma shows that $K$ is a rather stable relation.
We show that if we extend a frame $G$ to a frame $F$ such that
from worlds in $F-G$ we cannot reach worlds in $G$,
then $K$ on $G$ does not change.

\begin{lemma}\label{lemm:kfg}
Let $F=\langle W,R,S,\nu\rangle$ be a quasi--\ilmo--frame.
And let $G=\langle W^-,R^-,S^-,\nu^-\rangle$ be a subframe of $F$
(which means $W^-\subseteq W$, $R^-\subseteq R$, $S^-\subseteq S$ and $\nu^-\subseteq\nu$).
If
\renewcommand{\theenumi}{(\alph{enumi})}
\renewcommand{\labelenumi}{\theenumi}
\begin{enumerate}
\item\label{ass1:lemm:kfg} for each $f\in W-W^-$ and $g\in W^-$ not $f(R\cup S)g$ and
\item\label{ass2:lemm:kfg} $R{\upharpoonright_{W^-}}\subseteq K(G)$.
\end{enumerate}
Then $K(G)=K(F){\upharpoonright_{W^-}}$.
\end{lemma}
\begin{proof}
\newcommand{\kf}{K(F)}
\newcommand{\kg}{K(G)}
\newcommand{\kfg}{K(F){\upharpoonright_{W^-}}}

Clearly $\kfg$ satisfies the properties \ref{item0:defi:k}, \ref{item1:defi:k} and \ref{item2:defi:k}
of the definition of $\kg$ (Definition \ref{defi:k}). Thus, since $K_G$ is the smallest such relation,
we get that $\kg\subseteq \kfg$.

Let $K' = \kf - (\kfg - \kg)$.
Using Lemma \ref{lemm:kstep} one can show that $\kf\subseteq K'$.
From this it immediately follows that
$\kfg\subseteq \kg$.
\end{proof}
We finish the basic preliminaries with a somewhat complicated
variation of Lemma \ref{lemm:deficiencies}.

\begin{lemma}\label{lemm:ilmoexistence}

Let $\Gamma$ and $\Delta$ be \mcs's. $\Gamma\crit{C}\Delta$.
\[
P\rhd Q,S_1\rhd T_1,\ldots,S_n\rhd T_n\in\Gamma \ \ \mbox{ and }\ \ \ 
%\]
%and
%\[
\Diamond P\in\Delta.
\]
There exist $k\leq n$. \mcs's $\Delta_0,\Delta_1,\ldots,\Delta_k$ such that
\begin{itemize}
\item
Each $\Delta_i$ lies $C$-critical above $\Gamma$,
\item
Each $\Delta_i$ lies $\boxin$ above $\Delta$ (i.e.\ $\Delta\boxin\Delta_i$),

\item
$Q\in\Delta_0$,

\item
For all $1\leq j\leq n$, $S_j\in\Delta_h\Rightarrow \textrm{for some $i\leq k$, }T_j\in\Delta_i$.

\end{itemize}
\end{lemma}

\begin{proof}

\newcommand{\notss}{\ensuremath{\overline S}}
\newcommand{\bb}{\ensuremath{\neg B\wedge\Box\neg B}}
\newcommand{\notbb}{\ensuremath{B\vee\Diamond B}}
\newcommand{\dnota}{\ensuremath{\Diamond\neg A}}
\newcommand{\supto}[1]{S_{i_1}\vee\cdots\vee S_{i_{#1}}}

First a definition. For each $I\subseteq\{1,\ldots,n\}$ put
\[
\notss_I \equivbydef \bigwedge\{\neg S_i\mid i\in I\}.
\]
The lemma can now be formulated as follows. There exists $I\subseteq\{1,\ldots,n\}$ such that
\[
\{Q,\notss_I\}\cup\{\neg B,\Box\neg B\mid B\rhd C\in\Gamma\}\cup\{\Box A\mid\Box A\in\Delta\}\not\vdash\bot
\]
and, for all $i\not\in I$,
\[
\{T_i,\notss_I\}\cup\{\neg B,\Box\neg B\mid B\rhd C\in\Gamma\}\cup\{\Box A\mid\Box A\in\Delta\}\not\vdash\bot.
\]

So let us assume, for a contradiction, that this is false. Then there exist finite sets
$\mathcal{A}\subseteq\{A\mid\Box A\in\Delta\}$ and $\mathcal{B}\subseteq\{B\mid B\rhd C\in\Gamma\}$
such that, if we put
\[
A \equivbydef\bigwedge\mathcal{A},\mbox{ and } 
B \equivbydef\bigvee\mathcal{B},
\]
then, for all $I\subseteq\{1,\ldots,n\}$,
\begin{equation}\label{lem0_c1}
Q,\notss_I,\Box A,\bb\vdash\bot
\end{equation}
or,
\begin{equation}\label{lem0_c2}
\textrm{for some $i\not\in I$,}\quad T_i,\notss_I,\Box A,\bb\vdash\bot.
\end{equation}

We are going to define a permutation $i_1,\ldots,i_n$ of $\{1,\ldots,n\}$ such that if we
put $I_k = \{ i_j \mid j<k\}$ then
\begin{equation}\label{lem0_c5}
T_{i_k},\notss_{I_k},\Box A,\bb\vdash\bot.
\end{equation}
Additionally, we will verify that for each $k$
\[
\textrm{\eqref{lem0_c1} does not hold with $I_k$ for $I$}.
\]
We will define $i_k$ with induction on $k$. We define
$I_1=\emptyset$. And by Lemma \ref{lemm:deficiencies},
\eqref{lem0_c1} does not hold with $I=\emptyset$. Moreover,
because of this, \eqref{lem0_c2} must be true with $I=\emptyset$.
So, there exists some $i\in\{1,\ldots,n\}$ such that
\[
T_i,\Box A,\bb\vdash\bot.
\]
It is thus sufficient to take for $i_1$, for example, the least such $i$.

Now suppose $i_k$ has been defined.
We will first show that
\begin{equation}\label{lem0_c4}
Q,\notss_{I_{k+1}},\Box A,\bb\not\vdash\bot.
\end{equation}
Let us suppose that this is not so. Then
\begin{equation}\label{lem0_c7}
\vdash\Box(Q\rightarrow\dnota\vee\notbb\vee\supto{k}).
\end{equation}
So,
\begin{align*}
\Gamma\vdash P
&\rhd Q\\
&\rhd \dnota\vee\notbb\vee\supto{k-1}\vee S_{i_k}\ \ \ \ \ \ \ \ \ \textrm{by \eqref{lem0_c7}}\\
&\rhd \dnota\vee\notbb\vee\supto{k-1}\vee T_{i_k}\\
&\rhd \dnota\vee\notbb\vee\supto{k-1}\vee (T_{i_k}\wedge\Box A\wedge\bb\wedge\notss_{I_k})\\
&\rhd \dnota\vee\notbb\vee\supto{k-1}\ \ \ \ \ \ \ \ \ \ \ \ \ \ \ \ \ \textrm{by \eqref{lem0_c5}}\\
&\qquad\vdots\\
&\rhd \dnota\vee\notbb\vee S_{i_1}\\
&\rhd \dnota\vee\notbb\vee T_{i_1}\\
&\rhd \dnota\vee\notbb\vee (T_{i_1}\wedge\Box A\wedge\bb)\\
&\rhd \dnota\vee\notbb.\ \ \ \ \ \ \ \ \ \ \ \ \ \ \ \ \ \ \ \ \ \ \ \ \ \ \ \ \ \ \ \ \ \ \ \ \ \ \ \ \textrm{by \eqref{lem0_c5}, with $k=1$}.
\end{align*}
So by $\mo$,
\[
\Diamond P\wedge\Box A\rhd(\dnota\vee\notbb)\wedge\Box A \in\Gamma.
\]
But $\Diamond P\wedge\Box A\in\Delta$. So, by Lemma
\ref{lemm:deficiencies} there exists some $\mcs$ $\Delta$ with
$\Gamma\crit{C}\Delta$ that contains $\notbb$. This is a contradiction,
so we have shown \eqref{lem0_c4}.

But now, since \eqref{lem0_c4} is indeed true, and thus
\eqref{lem0_c1} with $I_{k+1}$ for $I$ is false, \eqref{lem0_c2}
must hold. Thus there must exist some $i\not\in I_{k+1}$ such that
\[
T_i,\notss_{I_{k+1}},\Box A,\bb\vdash\bot.
\]
So we can take for $i_{k+1}$, for example, the smallest such $i$.

It is clear that for $I=\{1,2,\dots,n\}$, \eqref{lem0_c2} cannot be true.
Thus, for $I=\{1,2,\ldots,n\}$, \eqref{lem0_c1} must be true. This implies
\begin{equation*}\label{lem0_c6}
\vdash\Box(Q\rightarrow\dnota\vee\notbb\vee\supto{n}).
\end{equation*}
Now exactly as above we can show $\Gamma\vdash P\rhd\dnota\vee\notbb$.
And again as above, this leads to a contradiction.
\end{proof}

In order to formulate the invariants needed in the main lemma
applied to \extil{M_0}, we need one more definition and a corollary.

\newcommand{\below}{\subset\xspace}
\newcommand{\belowone}{\subset_1\xspace}
\begin{definition}[$\belowone$, $\below$]
Let $F=\langle W,R,S,\nu\rangle$ be a quasi--frame.
Let $K=K(F)$. We define $\belowone$ and $\below$ as follows.
\begin{enumerate}
\item $x\belowone y\Leftrightarrow \exists w y' wKxK^1y'\trans{(\spure{w}{S})}y$
\item $x\below y\Leftrightarrow x(\belowone\cup K)^*y$
\end{enumerate}
\end{definition}

\begin{corollary}\label{lemm:ord0}
Let $F=\langle W,R,S,\nu\rangle$ be a quasi--frame. And let $K=K(F)$.
\begin{enumerate}
\item\label{i0:lemm:ord0}
$x\below y\wedge yKz\rightarrow xKz$
\item\label{i1:lemm:ord0}
If $F$ is a quasi--\ilmo--frame, then $x\below y\Rightarrow\nu(x)\boxin\nu(y)$.
\end{enumerate}
\end{corollary}

\subsection{Frame condition}\label{sect:ilmo:framecondition}

The following theorem is well known.
\begin{theorem}\label{theo:ilmoframe}
For an \il-frame $F=\langle W,R,S,\nu\rangle$ we have
\[ \forall wxyy'z\;(wRxRyS_wy'Rz\rightarrow xRz) \Leftrightarrow F\models \mo.\]
\end{theorem}

\subsection{Invariants}\label{sect:ilmo:invariants}

\newcommand{\invmain}[1]{\ensuremath{\mathcal{I}_{#1}}\xspace}
\newcommand{\invsub}[1]{\ensuremath{\mathcal{J}_{#1}}\xspace}

\newcommand{\invdef}{\invmain{\textrm{d}}}
\newcommand{\invlinbox}{\invmain{\Box}}
\newcommand{\invsbox}{\invmain{S}}
\newcommand{\invcback}{\invmain{\mathcal{N}}}
\newcommand{\invadd}{\invmain{\adset{D}}}
\newcommand{\invmo}{\invmain{\mo}}

\newcommand{\invone}{\invsub{\textrm{u}}}
\newcommand{\invtwo}{\invsub{K^1}}
\newcommand{\invswitch}{\invsub{\nu_1}}
\newcommand{\invs}{\invsub{\nu_2}}
\newcommand{\invnu}{\invsub{\nu_3}}
\newcommand{\invid}{\invsub{\nu_4}}
\newcommand{\invorder}{\invsub{\below}}
\newcommand{\invpres}{\invsub{\mathcal{N}_1}}
\newcommand{\invcsback}{\invsub{\mathcal{N}_2}}

\newcommand{\invlinboxfull}{for all $y$, $\{ \nu(x) \mid xK^1y \} \textrm{ is linearly ordered by }\boxin$}
\newcommand{\invdeffull}{$wK^1x\wedge wK^{\geq 2}x'(S_w\cup K)^*x\rightarrow
                                    \textrm{`there does not exists a deficiency in $w$ w.r.t. $x$'}$}
\newcommand{\invsboxfull}{   $wKxKy(S_w\cup K)^*y'\rightarrow$\\
            `the $\boxin$-max of $\{\nu(t)\mid wKtK^1y'\}$, if it exists,
            is $\boxin$-larger than $\nu(x)$'}
\newcommand{\invcbackfull}{$wKxKy\wedge y\in\ncone{A}{w}\rightarrow x\in\ncone{A}{w}$}
\newcommand{\invonefull}{$wK^{\geq 2}x\trans{(\spure{w}{S})}y\wedge wK^{\geq 2}x'\trans{(\spure{w}{S})}y\rightarrow x=x'$}
\newcommand{\invtwofull}{$wKxK^1y\trans{(\spure{w}{S})}y'K^1z\rightarrow xK^1z$}
\newcommand{\invswitchfull}{$\textrm{`$\nu(w,y)$ is defined'}\wedge vKy\rightarrow v\below w$}
\newcommand{\invsfull}{$\textrm{`$\nu(w,y)$ is defined'}\rightarrow wK^1y$}
\newcommand{\invorderfull}{$y\below x\wedge x\below y\rightarrow y=x$}
\newcommand{\invpresfull}{$x\trans{(\spure{v}{S})}y\wedge wKy\wedge x\in\ncone{A}{w}\rightarrow y\in\ncone{A}{w}$}
\newcommand{\invcsbackfull}{$x\trans{(\spure{w}{S})}y\wedge y\in\ncone{A}{w}\rightarrow x\in\ncone{A}{w}$}
\newcommand{\invnufull}{If $\nu(v,y)$ and $\nu(w,y)$ are defined then $w=v$}
\newcommand{\invidfull}{If $x\trans{(\spure{w}{S})}y$, then $\nu(w,y)$ is defined}
\newcommand{\invaddfull}{$xRy \rightarrow \exists \, A {\in} (\nu (y) \setminus \nu (x))\cap
                    \{ \Box D \mid D \in \adset{D} \}$}

Let $\adset{D}$ be some finite set of formulas, closed under
subformulas and single negation.
%First we list all the invariants.
%Then we explain them

During the construction we will keep track of the following main--invariants.
\begin{enumerate}
\item[\invlinbox]\invlinboxfull
\item[\invdef]\invdeffull
\item[\invsbox]\invsboxfull
\item[\invcback]\invcbackfull
\item[\invadd]\invaddfull
\item[\invmo] All conditions for an adequate \ilmo--frame hold
\end{enumerate}

In order to ensure that the main--invariants are preserved during the construction
we need to consider the following sub--invariants.\footnote{We call them
sub--invariants since they merely serve the purpose of showing that
the main-invariants are, indeed, invariant.}
\begin{enumerate}
\item[\invone]\invonefull
\item[\invtwo]\invtwofull
\item[\invorder] \invorderfull
\item[\invpres] \invpresfull
\item[\invcsback] \invcsbackfull
\item[\invswitch]\invswitchfull
\item[\invs]\invsfull
\item[\invid] \invidfull
\item[\invnu] \invnufull
\end{enumerate}

%Now let us explain some of these invariants. Let us assume that
%$wRy$, $A\rhd B\in\nu(w)$, $A\in\nu(y)$ and $y\in\ncone{C}{w}$. We
%want to find some $\Delta$ such that $B\in\Delta$,
%$\nu(w)\crit{C}\Delta$ and $wRtRy\rightarrow\nu(t)\boxin\Delta$.
%If the cardinality of the set $\{\nu(t)\mid wRtRy\}$ is larger
%than one it is not clear with which $\nu(t)$ we have to apply
%Lemma \ref{lemm:ilmoexistence}. \invlinbox assures that this set
%has a $\boxin$-maximum and we can take just that element. Another
%point of attention in this situation is the criticality of
%$\nu(y)$. We want to find a $\Delta$ that lies $C$ critical above
%$\nu(w)$. But we do not work with $\nu(y)$ but with some $t$ for
%which $wRtRy$. \invcback assures that $\nu(t)$ lies $C$-critical
%above $\nu(w)$.

What can we say about these invariants?
\invlinbox, \invsbox, \invcback and \invdef were discussed in Section \ref{sect:ilmo:diffi}.
%\invadd ensures the boundedness of all the frames. Which is needed
%to ensure the conversely well--foundedness of $R$ in the end.
\invmo is there to ensure that our final frame is an \ilmo--frame.
About the sub--invariants there is not much to say. They are merely technicalities
that ensure that the main--invariants are invariant.

Let us first show that if we have a quasi--\ilmo--frame that
satisfies all the invariants, possibly \invmo excluded, then we
can assume, nevertheless, that \invmo holds as well.

\begin{corollary}\label{coro:invilmoextension}
Any quasi--\ilmo--frame that satisfies all of the above invariants,
except possibly \invmo, can be extended to an \ilmo--frame that
satisfies all the invariants.
\end{corollary}

\begin{proof}
Only \invadd and \invdef need some attention. All the other
invariants are given in terms of relations that do not change
during the construction of the \ilmo-closure (Lemma
\ref{lemm:ilmoclosure}).
\end{proof}

\begin{lemma}\label{lemm:ncone1}
Let $F=\langle W,R,S,\nu\rangle$ be a quasi--\ilmo--frame. Then
$F\models x\in\ncone{A}{w}$ iff. one of the following cases
applies.
\begin{enumerate}
\item $\nu(w,x)=A$
\item There exists $t\in\ncone{A}{w}$ such that $tKx$
\item There exists $t\in\ncone{A}{w}$ such that $t\spure{w}{S}x$
\end{enumerate}
\end{lemma}

\begin{corollary}\label{lemm:ncone2}
Let $F$ be a quasi--\ilmo--frame that satisfies \invid.
Let $w,x\in F$ and let $A$ be a formula.
Then $x\in\ncone{A}{w}$ implies $\nu(w,x)=A$ or there exists some $t\in\ncone{A}{w}$ such that $tKx$.
\end{corollary}

\begin{lemma}\label{lemm:conelemma}
Let $F$ be a quasi--frame which satisfies \invcsback, \invswitch,
\invnu and \invid. Then
$x\spure{v}{S}y,\, y\in\ncone{A}{w}\Rightarrow x\in\ncone{A}{w}$.
\end{lemma}

\begin{proof}
Suppose $x\spure{v}{S}y$ and $y\in\ncone{A}{w}$. Then, by Corollary
\ref{lemm:ncone2}, $\nu(w,y)=A$ or, for some
$t\in\ncone{A}{w}$, $tKy$. In the first case
we obtain $w=v$ by \invnu and \invid. And thus by \invcsback,
$x\in\ncone{A}{w}$. In the second case we have, by \invid
and \invswitch that $t\below v$. Which implies, by Lemma
\ref{lemm:ord0}--\ref{i0:lemm:ord0}, $tKx$.
\end{proof}

\subsection{Solving problems}\label{sect:ilmo:problems}

% Point with problem
\newcommand{\problemp}{\const{a}}
% Point that solves the problem
\newcommand{\solve}{\const{b}}

Let $F=\langle W,R,S,\nu\rangle$ be a quasi--\ilmo--frame that
satisfies all the invariants. Let $(\problemp,\neg(A\rhd B))$ be a
\adset{D}-problem in $F$. We fix some $\solve\not\in W$. Using Lemma
\ref{lemm:problems} we find a \mcs $\Delta_\solve$, such that
$\nu(\problemp)\crit{B}\Delta_\solve$ and $A,\Box\neg
A\in\Delta_\solve$. We put
\begin{align*}
 \hat F & = \langle\hat W,\hat R,\hat S,\hat\nu\rangle \\
        & = \langle W\cup\{\solve\},
                R\cup\{\pair{\problemp,\solve}\},
                    S,
                    \nu\cup\{\pair{\solve,\Delta_\solve},\pair{\pair{\problemp,\solve},B}\}\rangle,
\end{align*}
and define $\hat K = K(\hat F)$.
The frames $F$ and $\hat F$ satisfy the conditions of Lemma
\ref{lemm:kfg}. Thus we have
\begin{equation}\label{prob:ilmoapp:k}
\forall xy{\in}F\; xKy\Leftrightarrow x\hat Ky.
\end{equation}
Since $\hat S{=}S$, this implies that all simple enough
properties expressed in $\hat K$ and $\hat S$ using only parameters from $F$ are
true if they are true with $\hat K$ replaced by $K$.

%%%%%%%%%%%%%%%%%%%%%%%%%%%%
%                          %
% Problems quasi-frame     %
%                          %
%%%%%%%%%%%%%%%%%%%%%%%%%%%%

\begin{claim}
$\hat F$ is a quasi--\ilmo--frame.
\end{claim}

\begin{proof}
A simple check of Properties (1.--5.) of Definition \ref{defi:quasiframes}
(quasi--frames) and Properties (6.--10.) of Definition \ref{defi:ilmoquasi}
(quasi--\ilmo--frames) and the remaining ones in Definition
\ref{defi:quasiframes} (quasi--frames).
Let us comment on two of them.

$x\hat Ky\rightarrow \hat\nu(x)\sucs\hat\nu(y)$
follows from Lemma \ref{lemm:kstep} and \eqref{prob:ilmoapp:k}.

Let us show $\hat F\models x\in\ncone{C}{w}\Rightarrow\hat\nu(w)\crit{C}\hat\nu(x)$.
  We have
$\forall xw{\in}F\;F\models x\in\mathcal{N}_w^C\Leftrightarrow \hat F\models x\in\mathcal{N}_w^C$.
  So we only have to consider the case $\hat F\models\solve\in\mathcal{N}_w^C$.
  If $w=\problemp$ then we are done by choice of
  $\hat\nu(\solve)$. Otherwise, by Lemma \ref{lemm:conelemma}, we have for some $x\in F$,
  $F\models x\in\mathcal{N}_w^C$ and $x\hat K\solve$.
  By the first property we proved, we get $\hat\nu(x)\sucs\hat\nu(\solve)$. 
So, since $\hat\nu(w)\crit{C}\hat\nu(x)$ we have
  $\hat\nu(w)\crit{C}\hat\nu(\solve)$.
\end{proof}

Before we show that $\hat F$ satisfies all the invariants we prove some lemmata.

\begin{lemma}\label{lemm:probl:k1}
If for some $x\neq\problemp$, $x\hat K^1\solve$. Then there exist
unique $u$ and $w$ (independent of $x$) such that $w K^{\geq
2}u\trans{(\spure{w}{S})}\problemp$.
\end{lemma}

\begin{proof}
If such $w$ and $u$ do not exists then $T=K\cup\{\problemp,\solve\}$ satisfies
the conditions of Lemma \ref{lemm:k1calc}.
In which case $xK^1\solve$ gives $xT\solve$ which implies $x=\problemp$.
The uniqueness of $w$ follows from \invnu and \invid.
The uniqueness of $u$ follows from \invone and the uniqueness of $w$.
\end{proof}

In what follows we will denote these $w$ and $u$, if they exist, by $\const w$ and $\const u$.

\begin{lemma}\label{lemm:probl:k3}
For all $x$. If $x\hat K^1\solve$ then $x\below\problemp$.
\end{lemma}
\begin{proof}
Let $K'=K\cup\{(x,\solve)\mid x\hat K\solve\wedge x\below\problemp\}$.
It is not hard to show that $K'$ satisfies the conditions of $T$ in Lemma
\ref{lemm:k1calc}.
\end{proof}

\begin{lemma}\label{lemm:probl:k2}
Suppose the conditions of Lemma \ref{lemm:probl:k1} are satisfied
and let $\const u$ be the $u$ asserted 
%(by that lemma) 
to exist.
Then for all $x\neq\problemp$, if $x\hat K^1\solve$, then $x K^1\const u$.
\end{lemma}

\begin{proof}

By Lemma \ref{lemm:probl:k3} we have $x\below\problemp$.
Let
\[
x=x_0(\belowone\cup K)x_1(\belowone\cup K)\cdots(\belowone\cup K)x_n=\problemp.
\]
First we show $x=x_0\belowone x_1\belowone\cdots\belowone x_n=\problemp$.
Suppose, for a contradiction, that for some $i<n$, $x_iKx_{i+1}$.
Then, by Lemma \ref{lemm:ord0}, $xKx_{i+1}K\solve$. So, $xK^{\geq 2}\solve$. A contradiction.
The lemma now follows by showing, with induction on $i$ and using $F\models\invtwo$,
that for all $i\geq 0$, $x_{n-(i+1)}K^1\const u$.

\end{proof}

%%%%%%%%%%%%%%%%%%%%%%%%%%%%
%                          %
% Problems  sub-invariants %
%                          %
%%%%%%%%%%%%%%%%%%%%%%%%%%%%

\begin{lemma}
$\hat F$ satisfies all the sub-invariants.
\end{lemma}

\begin{proof}
We only comment on $\invtwo$ and $\invswitch$.
Let $K=K(\hat F)$.

%\begin{enumerate}

%\item[\invtwo] $w\hat K x\hat K^1y\trans{(\spure{w}{\hat S})}y'\hat K^1z\rightarrow x\hat K^1z$
%\invtwofull

\invswitch follows from Lemma \ref{lemm:probl:k3}, so let us treat \invtwo.
Suppose $w\hat Kx\hat K^1y\trans{(\spure{w}{\hat S})}y'\hat K^1z$.
We can assume that at least one of $w,x,y,y',z$ is not in $F$ and
the only candidate for this is $z$. So we have $z=\solve$. We can
assume that $x\neq y'$ (otherwise we are done at once), so the
conditions of Lemma \ref{lemm:probl:k1} are fulfilled and thus
$\const w$ and $\const u$ as stated there exist.

Suppose now, for a contradiction, that for some $t$, $x\hat Kt\hat K^1\solve$. Then by Lemma
\ref{lemm:probl:k2}, $t=\problemp$ or $t\hat K^1\const u$.
Suppose we are in the case $t=\problemp$.
Since $\nu(\const w,\const a)$ is defined and $x\hat K\const a$ we obtain by
\invswitch, that $x\below\const w$. Since $\const w\hat K^{\geq 2}\const u$ we obtain by
Lemma \ref{lemm:ord0} that $x\hat K^{\geq 2}\const u$.
In the case $t\hat K^1\const u$ we have $x\hat K^{\geq 2}\const u$ trivially.
So in any case we have
\[ x\hat K^{\geq 2}\const u. \]

However, by Lemma \ref{lemm:probl:k2} and since $y'\hat K^1z$ we have $y'\hat K^1\const u$ or $y'=\problemp$.
In the first case, since $F\models\invtwo$, we have $x\hat K^1\const u$.
In the second case we obtain, by the uniqueness of $\const u$, that $y=\const u$ and
thus $x\hat K^1\const u$. So in any case we have
\[ x\hat K^1\const u. \]
A contradiction.

%\item[\invswitch] $\textrm{`$\hat\nu(w,y)$ is defined'}\wedge v\hat Ky\rightarrow v\below w$

%\invswitchfull

%\end{enumerate}

\end{proof}

%%%%%%%%%%%%%%%%%%%%%%%%%%%%
%                          %
% Problems main-invariants %
%                          %
%%%%%%%%%%%%%%%%%%%%%%%%%%%%

\begin{lemma}
Possibly with the exception of \invmo, $\hat F$ satisfies all the main-invariants.
\end{lemma}

\begin{proof}
Let $K=K(\hat F)$. We only comment on \invlinbox and \invcback.
%\begin{enumerate}

First we treat \invlinbox.
So we have to show that for all $y$,
$\{\hat\nu(x)\mid x\hat K^1 y\}$ is linearly ordered by $\boxin$.
We only need to consider the case $y=\const \solve$.
If $\{\problemp\}=\{ x \mid x\hat K^1\solve \}$ then the claim is obvious.
So we can assume that the condition of Lemma \ref{lemm:probl:k1} is fulfilled
and we fix $\const u$ as stated.
The claim now follows by $F\models\invlinbox$ (with $y=\const u$) and noting that, by Lemma \ref{lemm:kstep},
$x\hat K^1\solve\Rightarrow x\boxin\problemp$.

Now we look at \invcback: $w\hat K x\hat K y\wedge \hat F\models y\in\ncone{A}{w}
\rightarrow \hat F\models x\in\ncone{A}{w}$.
Suppose $w\hat Kx\hat Ky$ and $\hat F\models y\in\ncone{A}{w}$. We
only have to consider the case $y=\solve$. Then, by Lemma
\ref{lemm:ncone1}, $\hat\nu(w,\solve)=A$ or for some
$t\in\ncone{A}{w}$ we have $t\spure{w}{\hat S}\solve$ or $t\hat
K^1\solve$. The first case is impossible by \invs. The second is
also clearly not so. Thus we have
\begin{equation}\label{eq0:clai:pinvsback}
t\hat K^1\solve.
\end{equation}
We suppose that the conditions of Lemma \ref{lemm:probl:k1} are fulfilled
(the other case is easy).
If $t\hat{K}^1\const u$ and $x\hat{K}^*\const u$ then we are done
simmilarly as the case above. So assume $t\hat{K}^1\problemp$ or
$x\hat{K}^*\problemp$. Since $wRt$ and $wRx$ in any case we have $w\hat{K}\problemp$.
Now by Lemma \ref{lemm:conelemma} and \invpres we have
$\const u\in\ncone{A}{w}\Leftrightarrow \problemp\in\ncone{A}{w}$.
Also, by \eqref{eq0:clai:pinvsback},
$\const u\in\ncone{A}{w}\vee\problemp\in\ncone{A}{w}$.
So since $x\hat K\const u$ or $x=\problemp$ or $x\hat K\problemp$
we obtain $x\in\ncone{A}{w}$ by $F\models\invcback$.

\end{proof}

To finish this subsection we note that by Lemma \ref{lemm:ilmoclosure}
and Corollary \ref{coro:invilmoextension}
we can extend $\hat F$ to an adequate \ilmo--frame that satisfies all invariants.
%Moreover, by Corollary \ref{coro:invilmoextension}, we know that this frame
%satisfies all the invariants.

\subsection{Solving deficiencies}\label{sect:ilmo:deficiencies}

%%%%%%%%%%%%%%%%%%%%%
%                   %
% Deficiencies      %
%                   %
%%%%%%%%%%%%%%%%%%%%%
Let $F=\langle W,R,S,\nu\rangle$ be an \ilmo--frame satisfing all
the invariants. Let \mbox{$(\const a,\const b,C\rhd D)$} be a
\adset{D}-deficiency in $F$.

Suppose $\const a R^{\geq 2}\const b$ (the case $\const a R^1\const b$ is easy).
Let $x$ be the $\boxin$-maximum of $\{x\mid\const aKxK^1\const b\}$.
This maximum exists by \invlinbox.
Pick some $A$ such that $\const b\in\ncone{A}{\const a}$.
(If such an $A$ exists, then by adequacy of $F$, it is unique.
If no such $A$ exists, take $A=\bot$.)
By \invcback and adequacy we have $\nu(\const a)\crit{A}\nu(x)$.
%(By Lemma \ref{lemm:botcrit}, $\crit{\bot}=\sucs$.)
So we have
$C\rhd D\in\nu(\const a)\crit{A}\nu(x)\ni\Diamond C$.
We apply Lemma \ref{lemm:ilmoexistence} to obtain,
for some set $Y$, disjoint from $W$, a set $\{\Delta_y\mid y\in Y\}$ of \mcs's with all the properties as stated in that lemma.
We define
\begin{align*}
\hat F =\langle&W\cup Y,
            R\cup\{\pair{\const a,y}\mid y\in Y\},\\
        &S\cup\{\pair{\const a,\const b,y}\mid y\in Y\}\cup\{\pair{\const a,y,y'}
\mid y,y'\in Y,y\neq y'\},\\
           &\nu\cup\{\pair{y,\Delta_y},\pair{\pair{\const a,y},A}\mid y\in Y\}\rangle.
\end{align*}

%%%%%%%%%%%%%%%%%%%%%%%
%                     %
% Deficiencies, quasi %
%                     %
%%%%%%%%%%%%%%%%%%%%%%%

\begin{claim}
$\hat F$ is a quasi--\ilmo--frame.
\end{claim}

\begin{proof}

An easy check of Properties (1.--5.) of Definition \ref{defi:quasiframes} (quasi--frames)
and Properties (6.--10.) of Definition \ref{defi:ilmoquasi} (quasi--\ilmo--frames).
Let us comment on two cases.

First we see that $x\hat Ky\rightarrow\hat\nu(x)\sucs\hat\nu(y)$.
We can assume $y\in Y$.
By Lemma \ref{lemm:kstep} we obtain some $z$ with $\hat\nu(x)\boxin\hat\nu(z)$ and $x(\hat R\cup \hat S)^*z\hat Ry$.
This $z$ can only be $\const a$.
By choice of $\hat\nu(y)$ we have $\hat\nu(\const a)\sucs\hat\nu(y)$.
And thus $\hat\nu(x)\sucs\hat\nu(y)$.

We now see that $w\hat Kx\hat Ky(\hat S_w\cup \hat K)^*y'\rightarrow\hat\nu(x)\boxin\hat\nu(y')$.
We can assume at least one of $w,x,y,y'$ is in $Y$.
The only candidates for this are $y$ and $y'$.
If both are in $Y$ then $w=\const a$ and an $x$ as stated does not exists.
So only $y'\in Y$ and thus in particular $y\neq y'$.
Now there are two cases to consider.

The first case is that for some
 $t$, $w\hat Kx\hat Ky(\hat S_w\cup \hat K)^*t \hat K y'$.
But, $\hat\nu(y')$ is $\boxin$-larger than $\hat\nu(t)$ by $x\hat Ky\rightarrow\hat\nu(x)\sucs\hat\nu(y)$.
Also we have $w Kx Ky( S_w\cup K)^*t$. So, $\hat\nu(x)=\nu(x)\boxin\nu(t)=\hat\nu(t)$.

The second case is
$w\hat Kx\hat Ky(\hat S_w\cup \hat K)^*\const b \hat S_w y'$.
In this case we have $w=\const a$.
$y'$ is chosen to be $\boxin$--larger than the
$\boxin$-maximum of $\{\nu(r)\mid \const aKrK^1\const b\}$.
We have $w Kx Ky( S_w\cup K)^*\const b$
So, by $F\models\invsbox$, this $\boxin$--maximum is $\boxin$--larger than $\nu(x)$.
\end{proof}

\begin{lemma}\label{lemm:deflinbox}
For any $x\in\hat F$ and $y\in Y$ we have
$x\hat K^1y\rightarrow x\below \const a$.
\end{lemma}

\begin{proof}
We put
$K'=K\cup\{(x,y)\mid y\in Y,\; x\hat Ky,\; x\below\const a\}$.
By showing that $K'$ satisfies the conditions of $T$ in Lemma \ref{lemm:k1calc}.
we obtain $x\hat K^1y\rightarrow xK'y$.
So if $x\hat K^1y$ then $xK'y$. But if $y\in Y$ then $xKy$ does not hold.
Thus we have $x\below\const a$.
\end{proof}

\begin{lemma}\label{lemm:defk1}
Suppose $y\in Y$ and $\const a \hat K^1 z$. Then for all $x$,
$x\hat K^1 y\rightarrow x\hat K^1 z$.
\end{lemma}

\begin{proof}
Suppose $x K^1 y$. By Lemma \ref{lemm:deflinbox} we have $x\below\const a$.
There exist $x_0,x_1,x_2,\ldots,x_n$ such that
$x=x_0(\belowone\cup K)x_1(\belowone\cup K)\cdots(\belowone\cup K)x_n=\const a$.
First we show that
$x=x_0\belowone x_1\belowone\cdots\belowone\const a$.
Suppose, for a contradiction that for some $i<n$, we have $x_iKx_{i+1}$.
Then $xKx_{i+1}Ky$ and thus $xK^{\geq 2}y$. A contradiction.
The lemma now follows by showing, with induction on $i$, using \invtwo, that for all $i\leq n$, $x_{n-i}K^1z$.
\end{proof}

%%%%%%%%%%%%%%%%%%%%%%%%%%%%%%%
%                             %
% Deficiencies sub-invariants %
%                             %
%%%%%%%%%%%%%%%%%%%%%%%%%%%%%%%

\begin{lemma}
$\hat F$ satisfies all the sub-invariants.
\end{lemma}

\begin{proof}The proofs are rather straightforward. We give two examples.

First we show \invone: $w\hat K^{\geq 2}x\trans{(\spure{w}{\hat S})}y\wedge w\hat K^{\geq 2}x'\trans{(\spure{w}{\hat S})}y
\rightarrow x=x'$.
Suppose that $w\hat K^{\geq 2}x\trans{(\spure{w}{\hat S})}y$ and $w\hat
K^{\geq 2}x'\trans{(\spure{w}{\hat S})}y$. We can assume that
$y\in Y$. (Otherwise all of $w,x,x',y$ are in $F$ and we are done
by $F\models\invone$.) We clearly have $w\in F$. If $x\in Y$ then
$w=\const a$ and thus $w\hat K^1x$. So, $x\not\in Y$. Next we show
that both $x,x'\neq\const b$.

Assume, for a contradiction, that at least one of them equals
$\const b$. W.l.o.g. we assume it is $x$. But then $wK^{\geq
2}\const b$ and $wK^{\geq 2}x'\trans{(\spure{w}{S})}\const b$. By
$F\models\invid$ we now obtain that $\nu(w,\const b)$ is defined.
And thus by $F\models\invs$, $wK^1\const b$. A contradiction.

So, both $x,x'\neq\const b$. But now $wK^{\geq
2}x\trans{(\spure{w}{S})}\const b$ and $wK^{\geq
2}x'\trans{(\spure{w}{S})}\const b$. So, by $F\models \invone$, we
obtain $x=x'$.

Now let us see that \invtwo holds, that is
$w\hat K x\hat K^1y\trans{(\spure{w}{\hat S})}y'\hat K^1
z\rightarrow x\hat K^1 z$.
Suppose $w\hat Kx\hat K^1y\trans{(\spure{w}{\hat S})}y'\hat K^1z$.
We can assume that $z\in Y$. (Otherwise all of $w,x,y,y',z$ are in
$F$ and we are done by $F\models\invtwo$.) Fix some $a_1\in F$ for
which $\const aK^1 a_1$. By Lemma \ref{lemm:defk1} we have
$y'K^1a_1$ and thus, since $F\models \invtwo$, $xK^1a_1$. By
definition of $\hat K$ we have $x\hat Kz$. Now, if for some $t$,
we have $x\hat Kt\hat K^1z$, then similarly as above,$tK^1a_1$.
So, this implies $xK^{\geq 2}a_1$. A contradiction, conclusion:
$xK^1z$.
\end{proof}

%%%%%%%%%%%%%%%%%%%%%%%%%%%%%%%%
%                              %
% Deficiencies main-invariants %
%                              %
%%%%%%%%%%%%%%%%%%%%%%%%%%%%%%%%

\begin{lemma}
Except for \invmo, $\hat F$ satisfies all main-invariants.
\end{lemma}

\begin{proof}
We only comment on \invlinbox and \invcback.

First we show \invlinbox:
For all $y$, $\{\hat\nu(x)\mid x\hat K^1y\}$ is linearly ordered by $\boxin$.
Let $y\in\hat F$ and consider the set $\{x\mid xK^1y\}$.
Since $\hat K\upharpoonright_{F}=K$ and for all $y\in Y$ there does not exists 
$z$ with $y\hat K^1z$ we only
have to consider the case $y\in Y$.
Fix some $a_1$ such that $\const a K^1 a_1K^*\const b$.
By Lemma \ref{lemm:deflinbox} for any such $y$ we have
\[ \{x\mid xK^1y\}\subseteq\{x\mid xK^1a_1\}.\]
And by $F\models \invlinbox$ with $a_1$ for $y$, we know that $\{\nu(x)\mid xK^1a_1\}$ is linearly
ordered by $\boxin$.

Now let us see \invcback:
$w\hat K x\hat K y\wedge \hat F\models y\in\ncone{A}{w}
\rightarrow \hat F\models x\in\ncone{A}{w}$.
Suppose $w\hat Kx\hat Ky$ $\hat F\models y\in\ncone{A}{w}$.
We can assume $y\in Y$.
By Lemma \ref{lemm:deflinbox}, $x\below\const a$.
So, $w KxK\const b$.
By Lemma \ref{lemm:conelemma}, $F\models\const b\in\ncone{A}{w}$ and 
thus $\hat F\models x\in\ncone{A}{w}$.
\end{proof}

To finish this section we noting that by Lemma \ref{lemm:ilmoclosure}
and  Corollary \ref{coro:invilmoextension}
we can extend $\hat F$ to an adequate \ilmo--frame that satisfies all invariants.
%Moreover, by Corollary \ref{coro:invilmoextension}, we know that this frame
%satisfies all the invariants.

\subsection{Rounding up}\label{sect:ilmo:roundingup}

%
%We have to show that the union of a bounded chain of frames that satisfy all the invariants is an \ilmo--frame.
%But the \ilmo--frame conditions are part of the invariants and
It is clear that the union of a bounded chain of \ilmo--frames is itself an \ilmo--frame.

\section{The logic \ilwstar}\label{sect:wstar}

In this section we are going to prove the following theorem.

\begin{theorem}
\ilwstar is a complete logic.
\end{theorem}
For a long time \ilwstar has been conjectured (\cite{Vis91}) to be \ilal. A first step in proving this conjecture would have been a modal completeness result. However, the modal completeness of \ilwstar resisted many attempts as the modal completeness of \ilmo, which is an essential part of \ilwstar, was so hard and involved. (In \cite{jonvelt99} a completeness proof for \ilw was given.) 

Finally, now that all the machinery has been developed, a modal completeness proof for \ilwstar can be given.
The completeness proof of \ilwstar lifts almost completely along
with the completeness proof for \ilmo. We only need some minor 
adaptations.

\subsection{Preliminaries}
The frame condition of \principel{W} is well known.
\newcommand{\ww}{\principel{W}\xspace}
\begin{theorem}\label{homo}
For any \il-frame $F$ we have that
$F\models \ww \Leftrightarrow \forall w\;(S_w\comp R)$ is conversely well-founded.
\end{theorem}

\noindent
We can define a new principle $\sf M_0^*$ that is equivalent to 
\principel{W^*}, as follows.
\[
{\sf M_0^*}: \ \ \  A \rhd B 
\rightarrow \Diamond A \wedge \Box C \rhd B \wedge \Box C \wedge  \Box \neg A
\]
\begin{lemma} \label{lemm:emnulsteriswester}
$\extil{M_0W} = \ilwstar = \extil{M_0^*}$
\end{lemma}

\begin{proof}
The proof we give consists of four natural parts.

First we see $\ilwstar \vdash {\sf M_0}$. We reason in \ilwstar 
and assume $A \rhd B$. Thus,
also $A \rhd (B \vee \Diamond A)$. Applying the $\sf W^*$ axiom to the latter
yields 
$(B \vee \Diamond A) \wedge \Box C \rhd (B \vee \Diamond A) 
\wedge \Box  C \wedge \Box \neg A$. From this we may conclude
\[
\begin{array}{lll}
\Diamond A \wedge \Box C & \rhd & (B \vee \Diamond A) \wedge \Box C \\
\ & \rhd & (B \vee \Diamond A) \wedge \Box C \wedge \Box \neg A\\
\ & \rhd & B \wedge \Box C
\end{array}
\]

Secondly, we see that $\ilwstar \vdash {\sf W}$. Again, we reason in 
\ilwstar. We assume $A \rhd B$ and take the $C$ in the $\sf W^*$ axiom
to be $\top$. Then we immediately see that 
$A\rhd B \rhd B \wedge \Box \top \rhd B \wedge \Box \top \wedge \Box \neg A 
\rhd B \wedge \Box \neg A$.

We now easily see that $\extil{M_0W}\vdash {\sf M_0^*}$. For, reason in \extil{M_0W} as
follows. By $\sf W^*$, $A\rhd B \rhd B \wedge \Box \neg A$. Now an application of
$M_0$ on $A \rhd B \wedge \Box \neg A$ yields 
$\Diamond A \wedge \Box C \rhd B \wedge \Box C \wedge \Box \neg A$.

Finally we see that $\extil{M_0^*}\vdash {\sf W^*}$. So, we reason in 
\extil{M_0^*} and assume $A\rhd B$. Thus, we have also
$\Diamond A \wedge \Box C \rhd B \wedge \Box C\wedge\Box\neg A$. We now conclude
$B \wedge \Box C \rhd B \wedge \Box C \wedge \Box \neg A$ easily as follows.
$B \wedge \Box C \rhd (B\wedge \Box C \wedge \Box \neg A) \vee 
(\Box C \wedge \Diamond A) \rhd B \wedge \Box C \wedge \Box \neg A$.
\end{proof}

\newcommand{\wstar}{\principel{W^*}\xspace}

\begin{corollary}\label{coro:framecondwstar}
For any \il-frame we have that $F\models\wstar$ iff. both
(for each $w$, $(S_w\comp R)$ is conversely well-founded) and
($\forall w, x, y, y', z \ (wRxRyS_wy'Rz\rightarrow xRz)$).
\end{corollary}

The frame condition of \wstar tells us how to correctly define the 
notions of adequate \ilwstar-frames and quasi-\ilwstar-frames.

\newcommand{\boxpin}[1]{\subsetneq^{#1}_\Box}
\begin{definition}[$\boxpin{\adset{D}}$]
Let $\adset{D}$ be a finite set of formulas. Let $\boxpin{\adset{D}}$ be a binary relation
on \mcs's defined as follows.
$\Delta\boxpin{\adset{D}}\Delta'$ iff.\
\begin{enumerate}
\item $\Delta\boxin\Delta'$,
\item For some $\Box A\in\adset{D}$ we have $\Box A\in\Delta'-\Delta$.
\end{enumerate}
\end{definition}

\begin{lemma}\label{lemm:zuigendlemma}
Let $F$ be a quasi-frame and $\adset{D}$ be a finite set of formulas.
If $wRxRyS_wy'\rightarrow\nu(x)\boxpin{\adset{D}}\nu(y')$ then $(R\comp S_w)$ is
conversely well-founded.
\end{lemma}
\begin{proof}
%Take $w\in F$ and let 
%$x_0(R\comp S_w)x_1(R\comp S_w)x_2\cdots$ be an $R\comp S_w$ chain in $F$.
%For all $i\geq 1$ we have $wRx_i$. Thus, for all $i\geq 1$, $x_i\boxpin{\adset{D}} x_{i+1}$.
%Since $\adset{D}$ is finite the chain must be finite as well.
By the finiteness of \adset{D}.
\end{proof}

\begin{lemma}
Let $F$ be a quasi-\ilmo-frame. If $wRxRyS_wy'\rightarrow\nu(x)\boxpin{\adset{D}}\nu(y')$
then $wRxRy(S_w\cup R)^*y'\rightarrow\nu(x)\boxpin{\adset{D}}\nu(y')$
\end{lemma}
\begin{proof}
Suppose $wRxRy(S_w\cup R)^*y'$.
$\nu(x)\boxpin{\adset{D}}\nu(y')$ follows with induction on the minimal number
of $R$-steps in the path from $y$ to $y'$.
\end{proof}

\begin{definition}[Adequate \extil{\sf{W}^*}-frame]
Let $\adset{D}$ be a set of formulas.
We say that an adequate \ilmo-frame is an adequate \extil{\sf{W}^*}-frame (w.r.t. $\adset{D}$) iff.\
the following additional property holds.
\begin{enumerate}
\addtocounter{enumi}{7}
%\item $wKxKy\trans{(\spure{w}{S})}y'\rightarrow x\boxpin{\adset{D}}y'$
\item $wRxRy\trans{(\spure{w}{S})}y'\rightarrow x\boxpin{\adset{D}}y'$
%\item For each $w$ we have that $R\comp S_w$ is conversely well-founded
\end{enumerate}
\end{definition}

\begin{definition}[Quasi-\extil{\sf{W}^*}-frame]\label{defi:ilwstarquasi}
Let $\adset{D}$ be a set of formulas.
We say that a quasi-\ilmo-frame is a quasi-\extil{\sf{W}^*}-frame (w.r.t. $\adset{D}$) iff.\
the following additional property holds.
\begin{enumerate}
\addtocounter{enumi}{12}
\item $wKxKy\trans{(\spure{w}{S})}y'\rightarrow x\boxpin{\adset{D}}y'$
\end{enumerate}
\end{definition}

In what follows we might simply talk of adequate \extil{\sf{W}^*}-frames and quasi-\extil{\sf{W}^*}
In these cases $\adset{D}$ is clear from context.

%\begin{corollary}
%For any adequate \extil{\sf{W}^*}-frame $F$ and 
%for each $w\in F$ we have that $(R\comp S_w)$ is conversely well-founded
%\end{corollary}

\begin{lemma}\label{lemm:ilwstarclosure}
Any quasi-\extil{\sf{W}^*}-frame can be extended to an
adequate \extil{\sf{W}^*}-frame. (Both w.r.t. the same set of formulas $\adset{D}$.)
\end{lemma}

\begin{proof}
Let $F$ be a quasi-\extil{\sf{W}^*}-frame. Then in particular
$F$ is a quasi-\ilmo-frame. So consider the proof of Lemma \ref{lemm:ilmoclosure}.
There we constructed a sequence of quasi-\ilmo-frames
$F=F_0\subseteq F_1\subseteq \bigcup_{i<\omega}F_i=\hat F$.
What we have to do, is to 
show that if $F_0(=F)$ is a quasi-\extil{\sf{W}^*}-frame,
then each $F_i$ is as well. Additionally we have to show
that $\hat F$ is an adequate \ilwstar-frame.

But this is rather trivial. As noted in the proof of Lemma \ref{lemm:ilmoclosure},
The relation $K$ and the relations $\trans{(\spure{w}{S})}$ are constant throughout
the whole process. So clearly each $F_i$ is a quasi-\extil{\sf{W}^*}-frame.

Also the extra property of quasi-\extil{\sf{W}^*}-frames is preserved under
unions of bounded chains. So, $\hat F$ is an adequate \extil{\sf{W}^*}-frame.
\end{proof}

\begin{lemma}\label{lemm:ilwstarexistence}
Let $\Gamma$ and $\Delta$ be \mcs's with $\Gamma\crit{C}\Delta$,
\[
P\rhd Q,S_1\rhd T_1,\ldots,S_n\rhd T_n\in\Gamma \ \  \mbox{ and }
\ \ \ \Diamond P\in\Delta.
\]
There exist $k\leq n$. \mcs's $\Delta_0,\Delta_1,\ldots,\Delta_k$ such that
\begin{itemize}
\item
Each $\Delta_i$ lies $C$-critical above $\Gamma$,
\item
Each $\Delta_i$ lies $\boxin$ above $\Delta$,
\item
$Q\in\Delta_0$,
\item
For each $i\geq 0$, $\Box\neg P\in\Delta_i$,
\item
For all 
$1\leq j\leq n$, $S_j\in\Delta_h\Rightarrow \textrm{for some $i\leq k$, }T_j\in\Delta_i$.
\end{itemize}
\end{lemma}

\begin{proof}
The proof is a straightforward adaptation of the proof of Lemma \ref{lemm:ilmoexistence}.
In that proof, a trick was to postpone an application of 
\principel{M_0} as long as possible. We do the same here but let
an application of \principel{M_0} on $P\rhd \Diamond P \vee \psi$ be 
preceded by an application of \ww to obtain $P\rhd \psi$.
\end{proof}

\subsection{Completeness}
Again, we specify the four ingredients from Remark \ref{rema:application}.
The {\bf Frame condition} is contained in Corollary \ref{coro:framecondwstar}.

\newcommand{\invwstar}{\ensuremath{\mathcal{I}_{w^*}}}
%{{wKxKy\trans{(\spure{w}{S})}y'\rightarrow x\boxpin{\adset{D}}y'}}
The {\bf Invariants} are all those of \ilmo and additionally
\begin{enumerate}
\item[$\invwstar$] $wKxKy\trans{(\spure{w}{S})}y'\rightarrow x\boxpin{\adset{D}}y'$
%The conditions for an adequate \extil{\sf{W}^*} (w.r.t. $\adset{D}$) hold
\end{enumerate}
Here, 
$\adset{D}$ is some finite set of formulas closed under subformulas and single negation.

{\bf Problems}.
We have to show that we can solve problems in an adequate \extil{\sf{W}^*}-frame
in such a way that we end up with a quasi-\extil{\sf{W}^*}-frame.
If we have such a frame then in particular it is an \ilmo-frame.
So, as we have seen we can extend this frame to a quasi-\ilmo-frame.
It is easy to see that whenever we started with an adequate \extil{\sf{W}^*}-frame
we end up with a quasi \extil{\sf{W}^*}-frame.
(This is basically Lemma \ref{lemm:ilwstarclosure}.)

{\bf Deficiencies}.
We have to show that we can solve any deficiency in an
adequate \extil{\sf{W}^*}-frame such that we end up with an
quasi-\extil{\sf{W}^*}-frame.
It is easily seen that the process as described in the case
of \ilmo works if we use Lemma \ref{lemm:ilwstarexistence} instead of 
Lemma \ref{lemm:ilmoexistence}.

{\bf Rounding up}.
We have to show that the union of a bounded chain of 
quasi-\ilwstar-frames that satisfy all the invariants
is an \extil{\sf{W}^*}-frame.
%All the \ilmo-frame conditions (which are part of the \extil{\sf{W}^*} conditions)
%are part of the invariants.
%It is clear that the union of a bounded chain of \ilmo-frames is itself an \ilmo-frame.
%So we only 
The only novelty is that we
have to show that in this union for each $w$ we have that  $(R\comp S_w)$
is conversely well-founded. But this is ensured by $\invwstar$ 
and Lemma \ref{lemm:zuigendlemma}.
%to be precise
%by the property $wRxRyS_wy'\rightarrow\nu(x)\boxpin{\adset{D}}\nu(y)$, in exactly the same
%manner as we can ensure the conversely well-foundedness of $R$ via the boundedness
%condition. Namely, in any $(R\comp S_w)$ chain the number of $\Box$-formulas
%from $\adset{D}$ increases. But $\adset{D}$ is finite. So, such a chain must be finite as well.

\section*{Acknowledgement}
We dedicate this series of three papers to Dick de Jongh. Dick supervised the Masters theses of both authors and suggested to study a step-by-step construction method to obtain modal completeness results. 

Furthermore, we thank Lev Beklemishev, Marta Bilkova, Rosalie Iemhoff, Pavel Pudl\'ak, Volodya Shavrukov and Albert Visser for questions, discussions and suggestions.

%\section{Modal incompleteness and arithmetic}
%\input{goris/pnulincomplete}
%\subsection{Generalized semantics}\label{subs:generalizedinmain}
%\input{goris/generalized}
%\subsection{Arithmetical soundness of \principel{R}}\label{subs:risvetsound}
%\input{goris/alleredelijke_sukkel}

\bibliographystyle{plain}
%\begin{thebibliography}{100}

%\end{thebibliography}
%\bibliography{goris/levref}

\end{document}